\documentclass[reqno]{amsart}
\usepackage{hyperref}

\def\N{{\mathbb{N}}}

\def\R{{\mathbb{R}}}
\def\E{{\mathbb{E}}}
\def\X{{\mathbb{X}}}
\def\Y{{\mathbb{Y}}}
\def\A{{\mathbb{A}}}
\def\G{{\mathbb{G}}}

\def\W{{\mathbb{W}}}
\def\S{{\mathbb{S}}}
\def\A{{\mathbb{A}}}
\def\G{{\mathbb{G}}}

\begin{document}
\title[envelope]{Envelope theorems for static optimization and Calculus of Variations}

\author[J. Blot, H. Yilmaz]
{Jo${\rm \ddot e}$l Blot, Hasan Yilmaz } 

\address{Jo\"{e}l Blot: Laboratoire SAMM EA4543,\newline
Universit\'{e} Paris 1 Panth\'{e}on-Sorbonne, centre P.M.F.,\newline
90 rue de Tolbiac, 75634 Paris cedex 13,
France.}
\email{blot@univ-paris1.fr}

\address{Hasan Yilmaz: Laboratoire LPSM, UMR 8001, \newline
Universit\'e de Paris, b\^atiment Sophie Germain, \newline
8 place Aur\'elie Nemours, 75013 Paris, France.}
\email{yilmaz.research@gmail.com}

\date{August, 23, 2021}

\begin{abstract}
We establish differentiability properties of the value function of problems of Static Optimization in an abstract infinite dimensional setting and we apply that to problems of Calculus of Variations. We lighten the assumptions of existing results, notably by using G\^ateaux and Hadamard differentials. Moreover we use recently established Multipliers Rules. 
\end{abstract}

\maketitle
\numberwithin{equation}{section}
\newtheorem{theorem}{Theorem}[section]
\newtheorem{lemma}[theorem]{Lemma}
\newtheorem{example}[theorem]{Example}
\newtheorem{remark}[theorem]{Remark}
\newtheorem{definition}[theorem]{Definition}
\newtheorem{proposition}[theorem]{Proposition}
\newtheorem{corollary}[theorem]{Corollary}
\section{Introduction}
The paper treats of envelope theorems for parameterized problems of Static Optimization as 
\[
(\mathcal{M}, \pi)
\left\{
\begin{array}{rl}
{\rm Maximize} & f(x, \pi)\\
{\rm subject} \; {\rm to} & x \in \G \\
\null & \forall i = 1,...,k, \; g_i(x, \pi) \geq 0 \\
\null & \forall j = 1,..,{\ell}, \; h_j(x, \pi) = 0,
\end{array}
\right.
\]
and for parameterized problems of Calculus of Variations (Lagrange problems) under parameterized constraints as
\[
(\mathcal{V}, \pi)
\left\{
\begin{array}{rl}
{\rm Maximize} & J(x,\pi) := \int_0^T L(t,x(t), x'(t), \pi) dt\\
{\rm subject} \; {\rm to} & x \in C^1([0,T], M), x(0) = a_0, x(T) = a_T\\
\null & \forall i = 1,...,k, \; G_i(x, \pi) := \int_0^T \mathfrak{g}_i(t, x(t), x'(t), \pi) dt \geq 0\\
\null & \forall j = 1,...,{\ell}, \; H_j(x, \pi) := \int_0^T \mathfrak{h}_j(t, x(t), x'(t), \pi) dt = 0.
\end{array}
\right.
\]
We precise all the elements of these problems in the following sections.\\
We denote by $V(\pi)$ the value function of one of these problems. An envelope theorem consists to provide conditions to ensure the differentiability of $V$ (in a meaning that we will specify) and to provide an expression of this differential.\\
When $\pi_0$ is a fixed value of the parameter, in an envelope theorem, it is assumed that the optimal solution exists for all the parameters which belong to a neighborhood of $\pi_0$, not only for $\pi_0$. This is the difference between the problem of the envelope theorem and the problem of the regular perturbations where the theory provides conditions to ensure the existence of solutions $x(\pi)$ when $\pi$ is near to $\pi_0$, \cite{BS}.\\
The envelope theorems are classical fundamental tools of the economic theory; cf. \cite{Bo}, \cite{Cap}, \cite{Car}, \cite{FP}, \cite{LB}, \cite{LL}, \cite{MS}, \cite{Si1}, \cite{Si2}, \cite{MRT} for example.
\vskip1mm
\noindent
Our contribution is to establish envelope theorems with lightened assumptions with respect to the existing results. Notably we don't assume the regularity of the multipliers which are associated to the solutions of these optimization problems, and we prove their continuity; their differentiability is unnecessary. We use the Hadamard differentiability instead of the continuous Fr\'echet differentiability that is weaker in infinite dimension. Note that to avoid to assume that the multipliers are of class $C^1$ is labelled as a conjecture in \cite{Bo} (p. 7).\\
To realize this aim, we use recent multipliers rules established in \cite{Yi} which contains an improvment of the results of \cite{Bl} and an improvment of clasical results in the setting Fr\'echet $C^1$ as exposed in \cite{ATF}, we establish new results on the functionals under an integral form, we provide a variation of the Gram-Schmidt method, we provide informations on the dual space of an useful subspace of the space of the continuously differentiable functions defined on a segment of $\R$.
\vskip1mm
\noindent
Now we describe the contents of the paper. In Section 2, we specify our notation, and we provide some preliminaries. In Section 3, we state an envelope theorem for the parameterized problems of Static Optimization in infinite dimension. \\
In Section 4, we prove the results of Section 3. In a first subsection, we establish results of topological algebra, we built a variation of the Gram-Schmidt method to permit us to prove the continuity of the multipliers (instead to assume it).\\
In Section 5, we establish an envelope theorem for the parameterized problems of Calculus of Variations. In a first subsection, we state an envelope theorem. In a second subsection, we establish new results on the differentiability of nonlinear integral functionals. In a third subsection, we establish results on the Euler equation. In a fourth subsection, we study dual spaces which are useful to our results, In the last subsection, we give a proof of our envelope theorem.
\section{Notation and Preliminaries}
$\X$ and $\Y$ are two real normed spaces.
\subsection{Notions of differentiability.} $\mathcal{L}(\X, \Y)$ denotes the space of the continuous linear mappings from $\X$ into $\Y$. The topological dual space of $\X$ is denoted by $\X^*$. Generally the norm of the dual spaces will be denoted by $\Vert \cdot \Vert_*$. When $\A \subset \X$, $C^0(\A, \Y)$ denotes the space of the continuous mappings from $\A$ into $\Y$.\\
%
Let $\G$ be an open subset of $\X$, $f : \G \rightarrow \Y$ be a mapping, $x \in \G$ and $v \in \X$. When it exists, the right-directional derivative (also called the right G\^ateaux variation) of $f$ at $x$ in the direction $v$ is $D^+_G f(x;v) := \lim\limits_{\theta \rightarrow 0+} \frac{1}{\theta}(f(x + \theta v) - f(x))$. When $D^+_G f(x;v)$ exists for all $v \in \X$ and when $D^+_Gf(x; \cdot) \in \mathcal{L}(\X, \Y)$, the G\^ateaux differential of $f$ at $x$ is $D_Gf(x) \in \mathcal{L}(\X,\Y)$ defined by $D_Gf(x) \cdot v := D^+_G f(x;v)$.\\
We say that $f$ possesses a Hadamard variation at $x$ for the increment $v$ when there exists $D^+_Hf(x;v) \in \Y$ such that, for all $(\theta_n)_n \in \; ] 0, + \infty[^{\N}$ converging to $0$, and for all $(v_n)_n \in \X^{\N}$ converging to $v$, we have $\lim\limits_{n \rightarrow + \infty} \frac{1}{\theta_n}(f(x + \theta_n v_n) - f(x)) = D^+_Hf(x;v)$. When $D^+_Hf(x;v)$ exists for all $v \in \X$, and when $D^+_H f(x; \cdot) \in \mathcal{L}(\X,\Y)$, the Hadamard differential of $f$ at $x$ is $D_H f(x) \in \mathcal{L}(\X,\Y)$ defined by $D_Hf(x) \cdot v := D^+_Hf(x;v)$.\\
When it exists, $D_Ff(x)$ denotes the Fr\'echet differential of $f$ at $x$. $C^1(\G, \Y)$ denotes the space of the continuously Fr\'echet differentiable mappings from $\G$ into $\Y$. When $T \in \; ]0, + \infty[$, $C^1([0,T], \G)$ denotes the space of the continuously differentiable functions from $[0,T]$ into $\G$, and when $\xi_0, \xi_T \in \X$, $C^1_{\xi_0,\xi_T}([0,T], \G)$ denotes the space of the functions $x \in C^1([0,T], \G)$ such that $x(0) = \xi_0$ and $x(T) = \xi_T$. $C^{\infty}_c([0,T], \R^n)$ denotes the space of the infinitely differentiable functions from $[0,T]$ into $\R^n$ such that their support is included into $\, ]0,T[$.\\
When $\X$ is a finite product of $m$ real normed spaces, $\X = \prod_{1 \leq i \leq m} \X_i$, if $k \in \{ 1,...,m \}$, $D_{F,k}f(x)$ (respectively $D_{H,k} f(x)$, respectively $D_{G,k} f(x)$) denotes the partial Fr\'echet (respectively Hadamard, respectively G\^ateaux) differential of $f$ at $x$ with respect to the $k$-th vector variable. If $1 \leq k_1 \leq k_2 \leq k_3 \leq m$, $D_{H, (k_1,k_2,k_3)} f(x)$ denotes the Hadamard differential of the mapping \\
$[(x_{k_1}, x_{k_2}, x_{k_3}) \mapsto f(x_1,.., x_{k_1}, ..., x_{k_2}, ..., x_{k_3}, ...,x_m)]$ at the point $(x_{k_1}, x_{k_2}, x_{k_3})$. We refer to \cite{Fl} for all these notions of differentiability.
\subsection{Bounded variation functions.} We consider the functions from $[0,T]$ into a normed vector space which are of bounded variation (cf. \cite{La2}(Chapter 10, Section 1). 
$BV([0,T], \X)$ denotes the set of such functions. $C^0_R(\; ]0,T[ \;, \X)$ denotes the set of the functions from $[0,T]$ into $\X$ which are right-continuous on \;$]0,T[$. We define $NBV([0,T], \X) := \{ g \in BV([0,T], \X) \cap C^0_R(\; ]0,T[ \;, \X) : g(0) = 0 \}$. A function of $NBV([0,T], \X)$ is called a normalized bounded variation function. $\Vert g \Vert_{BV} := V_0^T(g)$ defines a norm on $NBV([0,T], \X)$. If $\X$ is finite dimensional, if $(e_i)_{1 \leq i \leq d}$ is a basis of $\X$, and $(e_i^*)_{1 \leq i \leq d}$ is its dual basis, $g \in NBV([0,T], \X)$ if and only if $e_i^* \circ g \in NBV([0,T], \R)$, and so we can use the results on scalar-valued bounded variation functions as given in \cite{KF}, \cite{KA}.
When $d = dim \X < + \infty$, $f \in AC([a,b], \X)$ (AC means: absolutely continuity) if and only if $e_i^* \circ f \in AC([a,b], \R)$ 
for all $i \in \{1,...,d \}$, and so we can use the results on the scalar-valued absolutely continuous functions as given in \cite{KF}, \cite{BGH}.
\subsection{Integrals.} When $\E$ is a real Banach space, $a < b$ in $\R$, and $f : [a,b] \rightarrow \E$ is a function, the Riemann integral of $f$ on $[a,b]$ is written $\int_a^b f(t) \, dt$ (cf. \cite{Di}, chapter 8).\\
We denote by $\mathcal{B}([a,b])$ the Borel tribe of $[a,b]$, and by $\mathfrak{m}_1$ the  canonical positive measure of Borel on $\mathcal{B}([a,b])$, charactarized by $\mathfrak{m}_1([ \alpha, \beta[) = \beta- \alpha$ when $\alpha < \beta$ belong to $[a,b]$. When $f : [a,b] \rightarrow \R$ is a Borel function which is $\mathfrak{m}_1$-integrable, we say that $f$ is Borel integrable on $[a,b]$, and we denote its Borel integral by $\int_{[a,b]} f \, d \mathfrak{m}_1$. We denote the set of such functions by $\mathcal{L}^1([a,b], \mathcal{B}([a,b]), \mathfrak{m}_1; \R)$. Conformly to \cite{Ru}, we don't use the term "Lebesgue integral" since we don't use the (completed) Lebesgue tribe, but only the Borel tribe.\\
When $\E$ is finite dimensional, if $(e_i)_{1 \leq i \leq d}$ is a basis of $\E$, and $(e_i^*)_{1 \leq i \leq d}$ its dual basis, a function $f : [a,b] \rightarrow \E$ is said to be Borel integrable on $[a,b]$ if and only if $e_i^* \circ f \in \mathcal{L}^1([a,b], \mathcal{B}([a,b]), \mathfrak{m}_1; \R)$ for all $i \in \{1,...,n \}$. We denote by $\mathcal{L}^1([a,b], \mathcal{B}([a,b]), \mathfrak{m}_1; \E)$ the space of such functions.\\
\section{Static Optimization}
Let $\X$ and $\Y$ be two real normed spaces, and $\G$ be an open subset of $\X$. We consider the following functions $f : \G \times \Y \rightarrow \R$, $g_i : \G \times \Y \rightarrow \R$ for all $i \in \{1,...,k \}$, and $h_j : \G \times \Y \rightarrow \R$ for all $j \in \{ 1,...,{\ell} \}$. We denote by $V(\pi)$ the value of the problem $(\mathcal{M}, \pi)$ when $\pi \in \Y$. \\
Let $\pi_0 \in \Y$. We consider the following conditions.
\vskip1mm
\noindent
\underline{ {\bf Condition on the dual space $\X^*$.}}\\
{\bf (A{\sc dua})} There exists  $( \cdot \mid \cdot )_{\X^*} \in C^0(\X^* \times \X^*, \R)$ which is an inner product on $\X^*$.\\
\underline{ {\bf Conditions on the solutions.}}\\
{\bf (A{\sc sol}1)} There exists an open neighborhood $P$ of $\pi_0$ in $\Y$ such that, for all $\pi \in P$, there exists a solution $\underline{x}(\pi)$ of $(\mathcal{M}, \pi)$.\\
{\bf (A{\sc sol}2)} There exists $\varpi \in \Y$ s.t. $D^+_G\underline{x}(\pi_0; \varpi)$ exists.\\
{\bf (A{\sc sol}2-bis)} $\forall \varpi \in \Y$, $D^+_G\underline{x}(\pi_0; \varpi)$ exists.\\
{\bf (A{\sc sol}2-ter}) $\forall \pi \in P$, $\forall \varpi \in \Y$, $D^+_G\underline{x}(\pi; \varpi)$ exists.\\
\underline{{\bf Conditions on the functions of the criterion and of the constraints.}}\\
{\bf (A{\sc fon}1)} $\forall \psi \in \{ f \} \cup \{ g_i : 1 \leq i \leq k \} \cup \{ h_j : 1 \leq j \leq {\ell} \}$, $D_H \psi( \underline{x}(\pi_0), \pi_0)$ exists, and, for all $\pi \in P$, $D_{H,1} \psi(\underline{x}(\pi), \pi)$ exists and $[ \pi \mapsto D_{H,1} \psi(\underline{x}(\pi), \pi)]$ is continuous on $P$.\\
{\bf (A{\sc fon}2)}  $\forall \psi \in \{ f \} \cup \{ g_i : 1 \leq i \leq k \} \cup \{ h_j : 1 \leq j \leq {\ell} \}$, $\forall \pi \in P$, $D_{H}\psi(\underline{x}(\pi), \pi)$ exists and $[ \pi \mapsto D_{H,2} \psi( \underline{x}(\pi), \pi)]$ is continuous on $P$.\\
\underline{ {\bf Conditions on the constraints functions only.}}\\
{\bf (A{\sc con}1)} $\forall \pi \in P, \forall j \in \{1,...,{\ell} \}, [x \mapsto h_j(x,\pi)]$ is continuous on a neighborhood of $\underline{x}(\pi)$.\\
{\bf (A{\sc con}2)} $D_{H,1}g_1(\underline{x}(\pi_0), \pi_0)$,..., $D_{H,1}g_k(\underline{x}(\pi_0), \pi_0), D_{H,1}h_1(\underline{x}(\pi_0), \pi_0)$,...,\\
$D_{H,1}h_{\ell}(\underline{x}(\pi_0), \pi_0)$ are linearly independent.\\ 
\begin{theorem}\label{th31} {\bf (Envelope Theorem).} We assume that (A{\sc dua}), (A{\sc sol}1),\\
 (A{\sc sol}2), (A{\sc fon}1), (A{\sc con}1), (A{\sc con}2) are fulfilled. Then the following assertions hold.
\begin{itemize}
\item[{\bf (i)}] The right derivative of the value function of $(\mathcal{M}, \pi)$ at $\pi_0$ in the direction $\varpi$ exists and 
$D^+_GV(\pi_0; \varpi)  =  D_{H,2} f(\underline{x}(\pi_0), \pi_0) \cdot \varpi +$\\
$ \sum\limits_{1 \leq i \leq k} \lambda_i(\pi_0) D_{H,2} g_i(\underline{x}(\pi_0), \pi_0) \cdot \varpi
 + \sum\limits_{1 \leq j \leq {\ell}} \mu_j(\pi_0) D_{H,2} h_j(\underline{x}(\pi_0), \pi_0) \cdot \varpi$,\\
where $\lambda_1(\pi_0)$,..., $\lambda_k(\pi_0)$, $\mu_1(\pi_0)$, ..., $\mu_{\ell}(\pi_0)$ are the Karush-Kuhn-Tucker multipliers associated to the solution $\underline{x}(\pi_0)$ of $(\mathcal{M}, \pi_0)$.
\item[{\bf (ii)}] If in addition we assume (A{\sc sol}2-bis) instead of (A{\sc sol}2), then the value function is G\^ateaux differentiable at $\pi_0$, and 
$D_GV(\pi_0)  =  D_{H,2} f(\underline{x}(\pi_0), \pi_0) +$\\
$ \sum\limits_{1 \leq i \leq k} \lambda_i(\pi_0) D_{H,2} g_i(\underline{x}(\pi_0), \pi_0)
 + \sum\limits_{1 \leq j \leq {\ell}} \mu_j(\pi_0) D_{H,2} h_j(\underline{x}(\pi_0), \pi_0).$
\item[{\bf (iii)}] If in addition we assume that (A{\sc sol}2-ter) and (A{\sc fon}2) are fulfilled, then the value function is of class Fr\'echet $C^1$ on an open neighborhood $Q$ of $\pi_0$, and for all $\pi \in Q$, we have
$D_FV(\pi)  =  D_{H,2} f(\underline{x}(\pi), \pi) +$\\
$ \sum\limits_{1 \leq i \leq k} \lambda_i(\pi) D_{H,2} g_i(\underline{x}(\pi), \pi)
 + \sum\limits_{1 \leq j \leq {\ell}} \mu_j(\pi) D_{H,2} h_j(\underline{x}(\pi), \pi),$\\
where $\lambda_1(\pi)$,..., $\lambda_k(\pi)$, $\mu_1(\pi)$, ..., $\mu_{\ell}(\pi)$ are the Karush-Kuhn-Tucker multipliers associated to the solution $\underline{x}(\pi)$ of $(\mathcal{M}, \pi)$.

\end{itemize}
\end{theorem}
We can adapt this result to problems where the domains of the functions are affine sets instead of to be vector spaces.
Now we consider an affine subset $\A \subset \X$. The director vector subspace of $\A$ is denoted by $\S$. $\G^0$ is an open subset of $\A$. We consider the following functions. $f^0 : \G^0 \times \Y \rightarrow \R$, $g_i^0 : \G^0 \times \Y \rightarrow \R$ when $i \in \{1,...,k \}$, and $h_j^0 : \G^0 \times \Y \rightarrow \R$ when $j \in \{ 1,...,{\ell} \}$. With these elements, when $\pi \in \Y$, we build the following problem.
\[
(\mathcal{M}^0, \pi)
\left\{
\begin{array}{rl}
{\rm Maximize} & f^0(z, \pi)\\
{\rm subject} \; {\rm to} &  z \in \G^0\\
\null & \forall i = 1,...,k, \; g_i^0(z, \pi) \geq 0 \\
\null & \forall j = 1,..,{\ell}, \; h_j^0(z, \pi) = 0,
\end{array}
\right.
\]
$V^0(\pi)$ denotes the value of $(\mathcal{M}^0, \pi)$. Let $\pi_0 \in \Y$. We consider the following list of conditions.\\
\underline{{\bf Condition on the dual space $\S^*$.}}\\
{\bf ($A^0${\sc dua})} There exists  $( \cdot \mid \cdot )_{\S^*} \in C^0(\S^* \times \S^*, \R)$ which is an inner product on $\S^*$.\\
\underline{{\bf Conditions on the solutions.}}\\
{\bf ($A^0${\sc sol}1)} There exists an open neighborhood $P$ of $\pi_0$ in $\Y$ such that, for all $\pi \in P$, there exists a solution $\underline{z}(\pi)$ of $(\mathcal{M}^0, \pi)$.\\
{\bf ($A^0${\sc sol}2)} There exists $\varpi \in \Y$ s.t. $D^+_G\underline{z}(\pi_0; \varpi)$ exists.\\
{\bf ($A^0${\sc sol}2-bis)} $\forall \varpi \in \Y$, $D^+_G\underline{z}(\pi_0; \varpi)$ exists.\\
{\bf ($A^0${\sc sol}2-ter}) $\forall \pi \in P$, $\forall \varpi \in \Y$, $D^+_G\underline{z}(\pi; \varpi)$ exists.\\
\underline{{\bf Conditions on the functions of the criterion and of the constraints.}}\\
{\bf ($A^0${\sc fon}1)} $\forall \psi \in  \{ f^0 \} \cup \{ g_i^0 : 1 \leq i \leq k \} \cup \{ h_j^0 : 1 \leq j \leq {\ell} \}$, $D_H \psi( \underline{z}(\pi_0), \pi_0)$ exists, and, for all $\pi \in P$, $D_{H,1} \psi(\underline{z}(\pi), \pi)$ exists and $[ \pi \mapsto D_{H,1} \psi(\underline{z}(\pi), \pi)]$ is continuous on $P$.\\
{\bf ($A^0${\sc fon}2)}  $\forall \psi \in  \{ f^0 \} \cup \{ g_i^0 : 1 \leq i \leq k \} \cup \{ h_j^0 : 1 \leq j \leq {\ell} \}$, $\forall \pi \in P$, $D_{H}\psi(\underline{z}(\pi), \pi)$ exists and $[ \pi \mapsto D_{H,2} \psi( \underline{z}(\pi), \pi)]$ is continuous on $P$.\\
\underline{{\bf Conditions on the constraints functions only.}}\\
{\bf ($A^0${\sc con}1)} $\forall \pi \in P, \forall j \in \{1,...,{\ell} \}, [z \mapsto h_j^0(z,\pi)]$ is continuous on a neighborhood of $\underline{z}(\pi)$.\\
{\bf ($A^0${\sc con}2)} $D_{H,1}g_1(\underline{z}(\pi_0), \pi_0)$,..., $D_{H,1}g_k(\underline{z}(\pi_0), \pi_0), D_{H,1}h_1(\underline{z}(\pi_0), \pi_0)$,...,\\
$D_{H,1}h_{\ell}(\underline{z}(\pi_0), \pi_0)$ are linearly independent.\\ 
The differentials on the affine subset $\A$ are defined (when they exist) on the director vector subset $\S$ which is the tangent vector space at each point of $\A$.
\begin{corollary}\label{cor32}
We assume that ($A^0${\sc dua}), ($A^0${\sc sol}1), ($A^0${\sc sol}2), ($A^0${\sc fon}1),\\
($A^0${\sc con}1), ($A^0${\sc con}2) are fulfilled. Then the following assertions hold.
\begin{itemize}
\item[{\bf ($\alpha$)}] $D^+_G V^0(\pi_0; \varpi)$ exists and there exists $(\lambda_i(\pi_0))_{1 \leq i \leq k} \in \R_+^k$ and\\
 $(\mu_j(\pi_0))_{1 \leq j \leq {\ell}} \in \R^{\ell}$ such that $
D^+_G V^0(\pi_0; \varpi)  =  D_{H,2} f^0(\underline{z}(\pi_0), \pi_0) \cdot \varpi  +\\ \sum_{1 \leq i \leq k} \lambda_i(\pi_0) D_{H,2} g^0_i(\underline{z}(\pi_0), \pi_0) \cdot \varpi + \sum_{1 \leq j \leq {\ell}} \mu_j(\pi_0) D_{H,2} h^0_j(\underline{z}(\pi_0), \pi_0) \cdot \varpi$.
\item[{\bf ($\beta$)}] If in addition we assume ($A^0${\sc sol}2-bis) instead ($A^0${\sc sol}2), the value function is G\^ateaux differentiable at $\pi_0$ and we have
$D_G V^0(\pi_0)  =  D_{H,2} f^0(\underline{z}(\pi_0), \pi_0) \\ + \sum_{1 \leq i \leq k} \lambda_i(\pi_0) D_{H,2} g^0_i(\underline{z}(\pi_0), \pi_0) + \sum_{1 \leq j \leq {\ell}} \mu_j(\pi_0) D_{H,2} h^0_j(\underline{z}(\pi_0), \pi_0)$.
\item[{\bf ($\gamma$)}] If in addition we assume that ($A^0${\sc sol}2-ter) and ($A^0${\sc fon}2) are fulfilled, then the value function is of class Fr\'echet $C^1$ on an open neighborhood $Q$ of $\pi_0$, and, for all $\pi \in Q$, there exist $((\lambda_i(\pi))_{1\leq i\leq k},(\mu_j(\pi))_{1\leq j\leq {\ell}})\in \R^k_+\times \R^{\ell}$ s.t.
$D_F V^0(\pi)  =  D_{H,2} f^0(\underline{z}(\pi), \pi)  +\\
 \sum_{1 \leq i \leq k} \lambda_i(\pi) D_{H,2} g^0_i(\underline{z}(\pi), \pi) + \sum_{1 \leq j \leq {\ell}} \mu_j(\pi) D_{H,2} h^0_j(\underline{z}(\pi), \pi)$.
\end{itemize}
\end{corollary}
The following corollary only uses the Fr\'echet differentiability and is more easily accessible to users for concrete problems issued from modellings.
\begin{corollary}\label{cor33}
In the setting of Theorem \ref{th31}, under (A{\sc dua}), (A{\sc sol}1), (A{\sc sol}2-ter), (A{\sc con}2), we assume that the following condition is fulfilled.
\vskip1mm
\noindent $[ \pi \mapsto \underline{x}(\pi)]$ is continuous on $P$ and $\forall \psi \in \{f \} \cup \{ g_i : 1 \leq i \leq k \} \cup \{ h_j : 1 \leq j \leq {\ell} \}$, $\psi$ is Fr\'echet $C^1$ on $\G \times P$.
\vskip1mm
\noindent
Then (A{\sc fon}1), (A{\sc fon}2), (A{\sc con}1) are fulfilled, and all the conclusions of Theorem \ref{th31} hold, replacing the Hadamard differentials by Fr\'echet differentials in the formulas.
\end{corollary}
\begin{remark}\label{rem34} When $\X$ is finite dimensional, the conditions (A{\sc dua}) and ($A^0${\sc dua}) are automatically fulfilled. Moreover the Hadamard differentiability is equivalent to the Fr\'echet differentiability. Hence, without (A{\sc dua}), the statement of Theorem \ref{th31} holds and we can replace the Hadamard differentials by Fr\'echet differentials in the formulas of the conclusions.
\end{remark}
%

\section{Proofs of results of section 3}
\subsection{Results of topological algebra} $\E$ and $\W$ are two real normed spaces.
\begin{lemma}\label{lem41} Let $w_0 \in \W$ and $W$ be a neighborhood of $w_0$ in $\W$. Let $(b_i)_{1 \leq i \leq m} \in C^0(W, \E)^m$ s.t. $b_1(w_0)$, ..., $b_m(w_0)$ are linearly independent. Then there exists a neighborhood $W_0$ of $w_0$, $W_0 \subset W$, s.t. for all $w \in W_0$, $b_1(w)$, ..., $b_m(w)$ are linearly independent.
\end{lemma}
\begin{proof} For all $w \in W$, we consider the mapping $\Phi_w : \R^m \rightarrow \E$ defined by\\
 $\Phi_w(u^1,...,u^m) := \sum_{1 \leq i \leq m} u^i b_i(w)$ when $(u^1,...,u^m) \in \R^m$. Clearly we have $\Phi_w \in \mathcal{L}(\R^m, \E)$, and the linear independence of $b_1(w)$, ..., $b_m(w)$ is equivalent to have $Ker \Phi_w = \{ 0 \}$. This last equality is equivalent to have $\Phi_w^{-1}(\{ 0 \}) \cap S(0,1) = \{ 0 \}$, where $S(0,1)$ is the unit sphere of $\R^m$. Proceeding by contradiction, we assume: $\forall k \in \N_*$, $\exists w_k \in B(w_0, \frac{1}{k})$, $\exists v_k \in S(0,1)$ s.t. $\Phi_{w_k}(v_k) = 0$. Using the compactness of S(0,1) and the Weierstras-Bolzano theorem, we obtain the existence of an increasing function $\sigma : \N_* \rightarrow \N_*$ and of $u \in S(0,1)$ s.t. $\lim_{k \rightarrow + \infty} v_{\sigma(k)} = u$. Hence, we have, for all $k \in \N_*$, $\sum_{1 \leq i \leq m} v^i_{\sigma(k)} b_i(w_{\sigma(k)}) = 0$. Since the $b_i$ are continuous, we obtain $\sum_{1 \leq i \leq m}u^i b_i(w_0)= 0$, and the linear independence of the $b_i(w_0)$ implies that $u = 0$ which contredicts $u \in S(0,1)$.
\end{proof}
In the two following lemmas, we assume the existence of functional $( \cdot \mid \cdot)_{\E} \in C^0( \E \times  \E ,\R)$ which is an inner product on $\E$. The following lemma is a version of the Gram-Schmidt's process (cf. e.g. \cite{La1} p. 366) in presence of a parameter.
\begin{lemma}\label{lem42}
Let $W$ be an open subset of $\W$, and $(e_i)_{1 \leq i \leq n} \in C^0(W, \E)^n$ s.t. $e_1(w)$, ..., $e_n(w)$ are linearly independent for all $w \in W$.\\
Then there exists $(\epsilon_i)_{1 \leq i \leq n} \in C^0(W,\E \setminus \{0 \})^n$ s.t. the following properties hold.
\begin{itemize}
\item[{\bf (a)}] $\forall w \in W$, $\forall i \in \{1,...n \}$, $\epsilon_i(w) \in span \{e_1(w),...,e_i(w) \}$.
\item[{\bf (b)}] $\forall w \in W$, $\forall (i,j) \in \{1,...n \}^2$ s.t. $i \neq j$, $(\epsilon_i(w) \mid \epsilon_j(w) )_{\E} = 0$.
\item[{\bf (c)}] $\forall (i,j) \in \{1,...n \}^2$, $\exists \alpha^i_j \in C^0(W,\R)$ s.t., $\forall w \in W$, \\
$\epsilon_j(w) = \sum_{1 \leq i \leq n} \alpha^i_j(w) e_i(w)$, where $\alpha^j_j(w) = 1$ and $\alpha^i_j(w) = 0$ when $i> j$.
\end{itemize}
\end{lemma}
\begin{proof} We proceed by induction on $n$. When $n = 1$, we set $\epsilon_1 := e_1$, and we have $\alpha_1^1(w) = 1$ for all $ w \in W$, and the conclusions (a, b, c) are fulfilled. We do the assumption of induction on $k$, and we prove the result for $k+1$. We set
\begin{equation}\label{eq41}
\epsilon_{k + 1}(w) := e_{k+1}(w) - \sum_{1 \leq i \leq k} \frac{( e_{k+1}(w) \mid \epsilon_i(w))_{\E}}{\Vert \epsilon_i(w) \Vert_e^2} \epsilon_i(w)
\end{equation}
where $\Vert \cdot \Vert_e$ is the norm associated to $( \cdot \mid \cdot)_{\E}$. \\
Hence $\epsilon_{k+1}(w) \in span \{ e_{k+1}(w), \epsilon_1(w), ..., \epsilon_k(w) \} = span \{ e_j(w) : 1 \leq j \leq k+1 \}$ by using the induction assumption, and so (a) is fulfilled. Let $i < k+1$. We have $(\epsilon_{k+1}(w) \mid \epsilon_i(w))_{\E} =  (e_{k+1}(w) \mid  \epsilon_i(w))_{\E}$\\
$ - \sum_{1 \leq j \leq k, j \neq i}  \frac{( e_{k+1}(w) \mid \epsilon_j(w))_{\E}}{\Vert \epsilon_j(w) \Vert_e^2} (\epsilon_j(w) \mid \epsilon_i(w))_{\E}
 - \frac{(e_{k+1}(w) \mid \epsilon_i(w))_{\E}}{\Vert \epsilon_i(w) \Vert_e^2} (\epsilon_i(w) \mid \epsilon_i(w))_{\E}$\\
$ =  (e_{k+1}(w) \mid  \epsilon_i(w))_{\E} - (e_{k+1}(w) \mid  \epsilon_i(w))_{\E} = 0$.\\
Hence, (b) is fulfilled. Notice that $\epsilon_{k+1}(w) \neq 0$ for all $w \in W$, since, proceeding by contradiction, if there exists $w \in W$ s.t. $\epsilon_{k+1}(w) = 0$, then, using (\ref{eq41}), we obtain $e_{k+1}(w) \in span \{ e_i(w) : 1 \leq i \leq k \}$ which contradicts the linear independence of the $e_i(w)$. When $j \in \{ 1,...,k \}$, we inductively define
\begin{equation}\label{eq42}
\alpha^j_{k+1}(w) := - \sum_{1 \leq i \leq k} \frac{(e_{k+1}(w) \mid \epsilon_i(w))_{\E}}{\Vert \epsilon_i(w) \Vert_e^2} \alpha^j_i(w)
, \hskip4mm
\alpha_{k+1}^{k+1}(w) := 1
\end{equation}
and, when $j > k+1$, we set
\begin{equation}\label{eq43}
\alpha^j_{k+1}(w) := 0.
\end{equation}
Since $( \cdot \mid \cdot)_{\E}$, $\Vert \cdot \Vert_e^2$, $e_1$, ...,$e_k$, $\epsilon_1$, ..., $\epsilon_k$ are continuous, we see that $\alpha^j_{k+1}$ is continuous as a composition of continuous mappings, for all $j$. Doing a straightforward calculation, we obtain
\begin{equation}\label{eq44}
\epsilon_{k+1}(w) = \sum_{1 \leq j \leq k+1} \alpha^j_{k+1}(w) e_j(w) = \sum_{1 \leq j \leq n} \alpha^j_{k+1}(w) e_j(w).
\end{equation}
Hence the reasoning by induction is complete, and the proof of the lemma is complete.
\end{proof}
\begin{lemma}\label{lem43} Let $W$ be an open subset of $\W$ and $(e_i)_{1 \leq i \leq n} \in C^0(W, \E)^n$ s.t. $e_1(w)$, ..., $e_n(w)$ are linearly independent, for all $w \in W$. We set $F_w := span \{ e_i(w) : 1 \leq i \leq n \}$ for all $w \in W$, and $F := \bigcup\limits_{w \in W} (F_w \times \{ w \})$. When $(x,w) \in F$, $\mathfrak{x}^i(x,w)$ denotes the i-th coordinate of $x$ in the basis $(e_j(w))_{1 \leq j \leq n}$.\\
Then, for all $i \in \{1,...,n \}$, $ \mathfrak{x}i \in C^0(F, \R)$.
\end{lemma}
\begin{proof} We consider $(\epsilon_j)_{1 \leq j \leq n} \in C^0(W, \E)^n$ as defined in Lemma \ref{lem42}. When $w \in W$ and $x \in F_w$, we can write $x = \sum_{1 \leq i \leq n} y^i(x,w) \epsilon_i(w)$ where $y^i(x,w) := \frac{(x \mid \epsilon_i(w))_{\E}}{\Vert \epsilon_i(w) \Vert^2_e}$. Using Lemma \ref{lem42}, we know that $\epsilon_i \in C^0(W, \E \setminus \{0 \})$, and using the continuity of $( \cdot \mid \cdot )_{\E}$ with respect to the intial norm of $\E \times \E$, we obtain that $y^i \in C^0(F, \R)$ as a composition of continuous mappings.\\
Since $(e_i(w))_{1 \leq i \leq n}$ and $(\epsilon_i(w))_{1 \leq i \leq n}$ are two bases of the same vector space $F_w$, we can consider $M(w)$ the matrix of transition from $(e_i(w))_{1 \leq i \leq n}$ to $(\epsilon_i(w))_{1 \leq i \leq n}$. Using Lemma \ref{lem42}, the elements of the $i$-th column of $M(w)$ are: $\alpha^i_1(w)$, ..., $\alpha^i_{i-1}$, 1, 0, ...,0, and since the elements are continuous functions of $w$, we can assert that $[w \mapsto M(w)] \in C^0(W, GL(n,\R))$ where $GL(n, \R)$ is the classical group on the $n \times n$ invertible real matrices. The operator $\mathcal{I} : GL(n, \R) \rightarrow GL(n,\R)$, defined by $\mathcal{I}(L) := L^{-1}$, is well know to be continuous. Moreover $M(w)^{-1}$ is the matrix of transition from $(\epsilon_i(w))_{1 \leq i \leq n}$ to $(e_i(w))_{1 \leq i \leq n}$. Since we have $[\mathfrak{x}^1(x,w) ... \mathfrak{x}^n(x,w)]^t = M(w)^{-1} [y^1(x,w) ... y^n(x,w)]^t$, where the upper index $t$ denotes the transposition of the matrices, we obtain the continuity of the $\mathfrak{x}^i$ on $F$ as a composition of continuous mappings.
\end{proof}
\subsection{Proof of Theorem \ref{th31}} {\sf Conclusion (i)}.
We use Lemma \ref{lem41}, with $\W = \Y$, $\E = \X^*$ and (A{\sc con}2) to ensure the existence of an open neighborhood $Q$ of $\pi_0$ in $P$ such that, 
\begin{equation}\label{eq45}
\left.
\begin{array}{l}
\forall \pi \in Q, D_{H,1} g_1(\underline{x}(\pi), \pi), ..., D_{H,1} g_k (\underline{x}(\pi), \pi),
D_{H,1} h_1(\underline{x}(\pi), \pi),\\
 ..., D_{H,1} h_{\ell}(\underline{x}(\pi), \pi) \; 
{\rm are} \; {\rm linearly} \; {\rm independent}.
\end{array}
\right\}
\end{equation}
 We want to use the multipliers rule of \cite{Yi} (Theorem 2.2). Notice that our assumption (A{\sc fon}1) implies that assumptions (i) and (ii) of this multipliers rule are fulfilled. Our assumption (A{\sc fon}1) and Theorem 4.2.6 in \cite{Fl} imply that the third assumption of this multipliers rule is fulfilled. (A{\sc con}1) and (A{\sc fon}1) imply that the last assumption of this multipliers rule is fulfilled. Hence using this multipliers rule and (\ref{eq45}), we can ensure the existence of multipliers $(\lambda_i(\pi))_{1 \leq i \leq k} \in \R^k$ and $(\mu_j(\pi))_{1 \leq j \leq {\ell}} \in \R^{\ell}$ which satisfy the following conditions for all $\pi \in Q$.
\begin{equation}\label{eq46}
\forall \pi \in Q, \forall i \in \{1,...,k \}, \lambda_i(\pi) \geq 0.
\end{equation}
\begin{equation}\label{eq47}
\forall \pi \in Q, \forall i \in \{1,...,k \}, \lambda_i(\pi) g_i(\underline{x}(\pi), \pi) = 0.
\end{equation}
\begin{equation}\label{eq48}
\left.
\begin{array}{l}
\forall \pi \in Q, \;\; D_{H,1}f(\underline{x}(\pi), \pi) +
  \sum_{1 \leq i \leq k}  \lambda_i(\pi) D_{H,1}g_i(\underline{x}(\pi), \pi)\\
	+ \sum_{1 \leq j \leq {\ell}} \mu_j(\pi) D_{H,1} h_j(\underline{x}(\pi), \pi) = 0.
\end{array}
\right\}
\end{equation}
The following result establishes the continuity of the Karush-Kuhn-Tucker multipliers with respect to the parameter $\pi$.
\begin{lemma}\label{lem44}
For all $i \in \{1,...,k \}$, $\lambda_i \in C^0(Q,\R_+)$ and, for all $j \in \{1,...,{\ell} \}$, $\mu_j \in C^0(Q, \R)$.
\end{lemma}
\begin{proof} Under (A{\sc fon}1), (A{\sc dua}), with $\W = \Y$ and $\E = \X^*$, we can use Lemma \ref{lem43}: we set $e_1(\pi) := D_{H,1} g_1(\underline{x}(\pi), \pi)$, ..., $e_k(\pi) := D_{H,1} g_k(\underline{x}(\pi), \pi)$, $e_{k+1}(\pi) := D_{H,1} h_1(\underline{x}(\pi), \pi)$, ..., $e_{k + {\ell}}(\pi) := D_{H,1} h_{\ell}(\underline{x}(\pi), \pi)$, and using (\ref{eq45}) and (\ref{eq48}), the assumptions of Lemma \ref{lem43} are fulfilled, and we can conclude that the coordinates $\lambda_i$ and $\mu_j$ are continuous.
\end{proof}
\noindent
Since we have $V(\pi)= f(\underline{x}(\pi), \pi)$ for all $\pi \in Q$, and since the Hadamard differential satisfies the Chain Rule (\cite{Fl}, (4.2.5) p.263), we can calculate
\begin{equation}\label{eq49}
D^+_G V(\pi_0; \varpi)  =  D_{H,1} f(\underline{x}(\pi_0), \pi_0) \cdot D^+_G \underline{x}(\pi_0; \varpi) + D_{H,2} f(\underline{x}(\pi_0), \pi_0) \cdot \varpi.
\end{equation}
Under (\ref{eq48}), we also have 
\begin{equation}\label{eq410}
\left.
\begin{array}{l}
D_{H,1} f(\underline{x}(\pi_0), \pi_0) \cdot D^+_G \underline{x}(\pi_0; \varpi) = \\
- \sum_{1 \leq i \leq k} \lambda_i(\pi_0) D_{H,1} g_i (\underline{x}(\pi_0), \pi_0) \cdot D^+_G \underline{x}(\pi_0; \varpi)\\
-\sum_{1 \leq j \leq {\ell}} \mu_j(\pi_0) D_{H,1}h_j (\underline{x}(\pi_0), \pi_0) \cdot D^+_G \underline{x}(\pi_0; \varpi).
\end{array}
\right\}
\end{equation}
If $\lambda_i(\pi_0) > 0$, using Lemma \ref{lem44}, we can assert that there exists a neighborhood $N$ of $\pi_0$ in $Q$ such that $\lambda_i(\pi) > 0$ for all $\pi \in N$, and using (\ref{eq47}) we obtain $g_i(\underline{x}(\pi), \pi) = 0$ for all $\pi \in N$. Differentiating with respect to $\pi$, we obtain 
$$D_{H,1} g_i (\underline{x}(\pi_0), \pi_0) \cdot D^+_G \underline{x}(\pi_0;\varpi) + D_{H,2} g_i (\underline{x}(\pi_0), \pi_0) \cdot \varpi = 0.$$
To subsume this reasoning, we write
\begin{equation}\label{eq411}
\left.
\begin{array}{l}
\lambda_i(\pi_0) > 0 \Longrightarrow \lambda_i(\pi_0) D_{H,1}g_i(\underline{x}(\pi_0), \pi_0) \cdot D^+_G \underline{x}(\pi_0;\varpi)\\
= - \lambda_i(\pi_0) D_{H,2}g_i(\underline{x}(\pi_0), \pi_0) \cdot \varpi.
\end{array}
\right\}
\end{equation}
Moreover the following assertion is clear.
\begin{equation}\label{eq412}
\left.
\begin{array}{l}
\lambda_i(\pi_0) = 0 \Longrightarrow \lambda_i(\pi_0) D_{H,1}g_i(\underline{x}(\pi_0), \pi_0) \cdot D^+_G \underline{x}(\pi_0;\varpi)\\
= - \lambda_i(\pi_0) D_{H,2}g_i(\underline{x}(\pi_0), \pi_0) \cdot \varpi.
\end{array}
\right\}
\end{equation}
From (\ref{eq411}) and (\ref{eq412}), we obtain
\begin{equation}\label{eq413}
\left.
\begin{array}{l}
\forall i \in \{ 1,...,k \}, \; \lambda_i(\pi_0) D_{H,1}g_i(\underline{x}(\pi_0), \pi_0) \cdot D^+_G \underline{x}(\pi_0;\varpi)\\
= - \lambda_i(\pi_0) D_{H,2}g_i(\underline{x}(\pi_0), \pi_0) \cdot \varpi.
\end{array}
\right\}
\end{equation}
Since $h_j(\underline{x}(\pi), \pi) = 0$ when $\pi \in Q$, differentiating at $\pi_0$, we obtain the following assertion.
\begin{equation}\label{eq414}
\left.
\begin{array}{l}
\forall j \in \{ 1,...,{\ell} \}, \mu_j(\pi_0) D_{H,1} h_j(\underline{x}(\pi_0), \pi_0) \cdot D^+_G \underline{x}(\pi_0; \varpi)\\
= - \mu_j(\pi_0) D_{H,2} h_j(\underline{x}(\pi_0), \pi_0) \cdot \varpi.
\end{array}
\right\}
\end{equation}
Using (\ref{eq413}) and (\ref{eq414}) into (\ref{eq410}), we have
\begin{equation}\label{eq415}
\left.
\begin{array}{l}
D_{H,1}f(\underline{x}(\pi_0), \pi_0)  \cdot D^+_G \underline{x}(\pi_0;\varpi) = 
\sum_{1 \leq i \leq k} \lambda_i(\pi_0) D_{H,2}g_i(\underline{x}(\pi_0), \pi_0) \cdot \varpi  \\
+ \sum_{1 \leq j \leq {\ell}} \mu_j(\pi_0) D_{H,2} h_j (\underline{x}(\pi_0), \pi_0) \cdot \varpi.
\end{array}
\right\}
\end{equation}
Using (\ref{eq415}) into (\ref{eq49}), we obtain: $D^+_G V(\pi_0; \varpi)  =  D_{H,2}f(\underline{x}(\pi_0), \pi_0)  \cdot \varpi$ \\
$+ \sum_{1 \leq i \leq k}  \lambda_i(\pi_0) D_{H,2}g_i(\underline{x}(\pi_0), \pi_0) \cdot \varpi  
 + \sum_{1 \leq j \leq {\ell}} \mu_j(\pi_0) D_{H,2} h_j (\underline{x}(\pi_0), \pi_0) \cdot \varpi$,
and so the first conclusion of Theorem \ref{th31} is proven.\\
{\sf Conclusion (ii)}. Since (A{\sc sol}2-bis) implies (A{\sc sol}2), we can use the formula of $D^+_G V(\pi_0, \varpi)$ for all the increments $\varpi$; from this formula we see that \\
$[ \varpi \mapsto D^+_G V(\pi_0, \varpi)]$, from $\Y$ into $\R$, is a linear combination of elements of $\Y^*$. Consequently $V$ is G\^ateaux differentiable at $\pi_0$, and so the second conclusion of Theorem \ref{th31} is proven.\\
{\sf Conclusion (iii)}. From (A{\sc sol}2-ter), using the second conclusion at $\pi \in Q$ instead of $\pi_0$, we know that, for all $\pi \in Q$, we have $
D_GV(\pi) = D_{H,2} f(\underline{x}(\pi), \pi) + \sum_{1 \leq i \leq k} \lambda_i(\pi) D_{H,2} g_i(\underline{x}(\pi), \pi) +
\sum_{1 \leq j \leq {\ell}} \mu_j(\pi) D_{H,2} h_j(\underline{x}(\pi), \pi)$. \\
Using (A{\sc fon}2), we see that $D_GV \in C^0(Q, \Y^*)$, hence (\cite{Fl}, (4.4.7) Corollary 2, p.257), $V$ is Fr\'echet differentiable on $Q$ and $D_F V \in C^0(Q, \Y^*)$. The proof of Theorem \ref{th31} is complete.
\subsection{Proof of Corollary \ref{cor32}}
The strategy to realize this proof is to translate $(\mathcal{M}^0, \pi)$ into a new problem $(\mathcal{M}^1, \pi)$ which is a special case of $(\mathcal{M}, \pi)$, to apply Theorem \ref{th31} on $(\mathcal{M}^1, \pi)$, and to translate the conclusions on $(\mathcal{M}^1, \pi)$ into conclusions on $(\mathcal{M}^0, \pi)$.\\
For all $\psi^0 \in \{ f^0 \} \cup \{ g_i^0 : 1 \leq i \leq k \} \cup \{ h_j^0 : 1 \leq j \leq {\ell} \}$ we introduce the function $\psi : (\G^0 - \underline{z}(\pi_0)) \times \Y \rightarrow \R$ by setting $\psi(x, \pi) := \psi^0(x + \underline{z}(\pi_0), \pi)$. Note that $\G^0 - \underline{z}(\pi_0) $ is an open subset of $\S$. We consider the new following problem.
\[
(\mathcal{M}^1, \pi)
\left\{
\begin{array}{cl}
{\rm Maximize} & f(x, \pi)\\
{\rm subject} \; {\rm to} & x \in \G^0 - \underline{z}(\pi_0)\\
\null & \forall i \in \{ 1,...,k \}, \; g_i(x, \pi) \geq 0 \\
\null & \forall j \in \{ 1,..., {\ell} \}, \; h_j(x,\pi) = 0.
\end{array}
\right.
\]
This problem is a special case of $(\mathcal{M}, \pi)$ with $\X = \S$. It is easy to verify that $x$ is admissible for $(\mathcal{M}^1, \pi)$ if and only if $x + \underline{z}(\pi_0)$ is admissible for $(\mathcal{M}^0, \pi)$, and $x$ is a solution of $(\mathcal{M}^1, \pi)$ if and only if $x + \underline{z}(\pi_0)$ is a solution of $(\mathcal{M}^0, \pi)$. We set $\underline{x}(\pi) := \underline{z}(\pi) - \underline{z}(\pi_0)$. After that, note that the assumptions on $(\mathcal{M}^0, \pi)$ were done to be translatable into assumptions of Theorem \ref{th31} on $(\mathcal{M}^1, \pi)$.
  Then the conclusions of Theorem \ref{th31} are valid on $(\mathcal{M}^1, \pi)$. Now it suffices to translate the conclusions on $(\mathcal{M}^1, \pi)$ into conclusions on $(\mathcal{M}^0, \pi)$.  To do this, it suffices to note that $D^+_G \underline{x}(\pi; \varpi) = D^+_G \underline{z}(\pi; \varpi)$ and, for all $\psi^0 \in \{ f^0 \} \cup \{ g_i^0 : 1 \leq i \leq k \} \cup \{ h_j^0 : 1 \leq j \leq {\ell} \}$, $D_{H,2} \psi^0(\underline{z}(\pi), \pi) = D_{H,2} \psi( \underline{x}(\pi), \pi)$. The proof of the corollary is complete.
\section{Calculus of Variations}
This section is divided into the following subsections. In a first subsection, we state an envelope theorem for $(\mathcal{V}, \pi)$ after to specify the assumptions. In a second subsection, we establish new results on functionals under an integral form. In a third subsection, we treat of the Euler-Lagrange equation. In the last subsection, we give a proof of the envelope theorem.
\subsection{An envelope theorem.} $X = \R^n$ and $Y$ is a real normed space. $T \in \; ]0, + \infty[$ and  $M$ is an open subset of $\R^n$. We consider functions $L : [0,T] \times M \times \R^n \times Y \rightarrow \R$, $\mathfrak{g}_i : [0,T] \times M \times \R^n \times Y \rightarrow \R$ for all $i \in \{ 1,...,k\}$, $\mathfrak{h}_j : [0,T] \times M \times \R^n \times Y \rightarrow \R$ for all $j \in \{1,...,{\ell} \}$. We fix $a_0, a_T \in M$.\\
We fix $\pi_0 \in Y$ and we consider the following list of conditions: \\
\underline{\bf Conditions on the solutions}\\
{\bf (B{\sc sol}1)} There exists an open neighborhood $P$ of $\pi_0$ in $Y$ s.t., for all $\pi \in P$, there exists a solution $x(\pi)$ of $(\mathcal{V}, \pi)$, and $[ \pi \mapsto x(\pi)]$ is continuous at $\pi_0$.\\
{\bf (B{\sc sol}2)} There exists $\varpi \in Y$ s.t. $D^+_Gx(\pi_0; \varpi)$ exists.\\
{\bf (B{\sc sol}2-bis)} For all $\varpi \in Y$, $D^+_G x(\pi_0; \varpi)$ exists.\\
{\bf (B{\sc sol}2-ter)} For all $\pi \in P$, for all $\varpi \in Y$, $D^+_Gx(\pi; \varpi)$ exists.\\
\underline{\bf Conditions on the integrand of the criterion and of the constraints}\\
{\bf (B{\sc int}1)} There exist $\rho > 0$ and $\xi \in \mathcal{L}^1([0,T], \mathcal{B}([0,T]), \R_+, \mathfrak{m}_1)$ s.t., for all $\psi \in \{ L\} \cup \{ \mathfrak{g}_i : 1 \leq i \leq k \} \cup \{ \mathfrak{h}_j : 1 \leq j \leq {\ell} \}$, for all $t \in [0,T]$, for all $u, u_1 \in B(x(\pi_0)(t), \rho)$, for all $v, v_1 \in B(x(\pi_0)'(t), \rho)$, for all $\pi, \pi_1 \in B(\pi_0,\rho), \\
\vert \psi (t,u,v, \pi) - \psi(t, u_1,v_1, \pi_1) \vert \leq \xi(t)( \Vert u - u_1 \Vert + \Vert v - v_1 \Vert + \Vert \pi - \pi_1 \Vert)$. \\
{\bf (B{\sc int}2)}  For all $\psi \in \{ L \} \cup \{ \mathfrak{g}_i : 1 \leq i \leq k \} \cup \{ \mathfrak{h}_j : 1 \leq j \leq {\ell} \}$, $\psi$ is continuous on $[0,T] \times M \times \R^n \times P$, and for all $t \in [0,T]$, for all $\pi \in P$,
\\ $D_{H, (2,3,4)} \psi(t, x(\pi)(t), x(\pi)'(t), \pi)$ exists.\\
{\bf (B{\sc int}3)} For all $\psi \in \{ L \} \cup \{ \mathfrak{g}_i : 1 \leq i \leq k \} \cup \{ \mathfrak{h}_j : 1 \leq j \leq {\ell} \}$, for all $t \in [0,T]$, $[ \pi \mapsto D_{H, (2,3)} \psi(t, x(\pi)(t), x(\pi)'(t), \pi)]$ is continuous from $B(\pi_0, \rho)$ into $(\R^{n*}, \Vert \cdot \Vert_*)^2$.\\
{\bf (B{\sc int}3-bis)} For all $\psi \in \{ L \} \cup \{ \mathfrak{g}_i : 1  \leq i \leq k \} \cup \{ \mathfrak{h}_j : 1 \leq j \leq {\ell} \}$,  for all $t \in [0,T]$, $[ \pi \mapsto D_{H,(2,3,4)} \psi(t, x(\pi)(t), x(\pi)'(t), \pi)]$ is continuous on $B(\pi_0, \rho)$.\\
{\bf (B{\sc int}4)} For all $\psi \in  \{ L \} \cup \{ \mathfrak{g}_i : 1  \leq i \leq k \} \cup \{ \mathfrak{h}_j : 1 \leq j \leq {\ell} \}$, for all $\pi \in B(\pi_0, \rho)$, 
$[t \mapsto D_{H, (2,3,4)} \psi(t, x(\pi)(t), x(\pi)'(t), \pi)]$ is measurable from  $([0,T], \mathcal{B}([0,T]))$ into $((\R^n \times \R^n \times Y)^*, \mathcal{B}((\R^n \times \R^n \times Y)^*)$.\\
\underline{\bf Conditions on the integrands of the constraints only}\\
{\bf (B{\sc con})} For all $\lambda = (\lambda_i)_{1 \leq i \leq k} \in \R^k$ and for all $\mu = (\mu_j)_{1 \leq j \leq {\ell}} \in \R^{\ell}$ s.t. $(\lambda, \mu) \neq (0,0)$, $x(\pi_0)$ is not a solution of the Euler equation in Dubois-Reymond form (\cite{IT} p.106) $D_{H,3} \psi_{\lambda, \mu}(t,x(t),x'(t), \pi_0) = \int_{[0,t]} D_{H,2} \psi_{\lambda, \mu}(t, x(t), x'(t), \pi_0) \;d \mathfrak{m}_1(t) + c_{\lambda, \mu}$ \; $\mathfrak{m}_1$-a.e. $t \in [0,T]$, where\\
 $\psi_{\lambda, \mu}(t,x,v, \pi) := \sum_{1 \leq i \leq k} \lambda_i \mathfrak{g}_i(t,x,v, \pi) + \sum_{1 \leq j \leq {\ell}} \mu_j \mathfrak{h}_j(t,x,v, \pi)$ and $c_{\lambda, \mu} \in \R^{n*}$ is a constant.\\
\begin{theorem}\label{th51} {\bf (Envelope Theorem).} We assume that (B{\sc sol}1), (B{\sc sol}2), \\
(B{\sc int}1), (B{\sc int}2), (B{\sc int}3) and (B{\sc con}) are fulfilled. Then the following assertions hold.
\begin{itemize}
\item[{\bf (I)}] The value function of $(\mathcal{V}, \pi)$ admits a right directional derivative at $\pi_0$ in the direction $\varpi$ exists and we have
\[
\begin{array}{cl}
D^+_G V(\pi_0; \varpi)  = & \int_{[0,T]} D_{H,4} L (t, x(\pi_0)(t), x(\pi_0)'(t), \pi_0) \cdot \varpi \; d\mathfrak{m}_1(t) +\\
\null & \sum_{1 \leq i \leq k} \lambda_i(\pi_0) \int_{[0,T]} D_{H,4} \mathfrak{g}_i(t, x(\pi_0)(t), x(\pi_0)'(t), \pi_0) \cdot \varpi \; d\mathfrak{m}_1(t) +\\
\null & \sum_{1 \leq j \leq {\ell}} \mu_j(\pi_0) \int_{[0,T]} D_{H,4} \mathfrak{h}_j(t, x(\pi_0)(t), x(\pi_0)'(t), \pi_0) \cdot \varpi \; d\mathfrak{m}_1(t)
\end{array}
\]
where $(\lambda_i(\pi_0))_{1 \leq i \leq k}$ and $(\mu_j(\pi_0))_{1 \leq j \leq {\ell}}$ are the Karush-Kuhn-Tucker mutipliers associated to the solution $x(\pi_0)$ of the problem $( \mathcal{V}, \pi_0)$.
\item[{\bf (II)}] 
If in addition we replace (B{\sc sol}2) by (B{\sc sol}2-bis), the value function is G\^ateaux differentiable at $\pi_0$ and for all $\varpi \in Y$, we have 
\[
\begin{array}{cl}
D_G V(\pi_0) \cdot \varpi = & \int_{[0,T]} D_{H,4} L (t, x(\pi_0)(t), x(\pi_0)'(t), \pi_0) \cdot \varpi \; d\mathfrak{m}_1(t) +\\
\null & \sum_{1 \leq i \leq k} \lambda_i(\pi_0) \int_{[0,T]} D_{H,4} \mathfrak{g}_i(t, x(\pi_0)(t), x(\pi_0)'(t), \pi_0) \cdot \varpi \; d\mathfrak{m}_1(t) +\\
\null & \sum_{1 \leq j \leq {\ell}} \mu_j(\pi_0) \int_{[0,T]} D_{H,4} \mathfrak{h}_j(t, x(\pi_0)(t), x(\pi_0)'(t), \pi_0) \cdot \varpi \; d\mathfrak{m}_1(t)
\end{array}
\]
\item[{\bf (III)}] If in addition we assume that (B{\sc sol}2-ter), (B{\sc int}3-bis), and (B{\sc int}4) are fulfilled, then the value function is of class Fr\'echet $C^1$ on an open neighborhood of $\pi_0$, and for all $\pi$ which belongs to this neighborhood, for all $\varpi \in Y$, we have
\[
\begin{array}{cl}
D_F V(\pi) \cdot \varpi = & \int_{[0,T]} D_{H,4} L (t, x(\pi)(t), x(\pi)'(t), \pi) \cdot \varpi \; d\mathfrak{m}_1(t) +\\
\null & \sum_{1 \leq i \leq k} \lambda_i(\pi) \int_{[0,T]} D_{H,4} \mathfrak{g}_i(t, x(\pi)(t), x(\pi)'(t), \pi) \cdot \varpi \; d\mathfrak{m}_1(t) +\\
\null & \sum_{1 \leq j \leq {\ell}} \mu_j(\pi) \int_{[0,T]} D_{H,4} \mathfrak{h}_j(t, x(\pi)(t), x(\pi)'(t), \pi) \cdot \varpi \; d\mathfrak{m}_1(t)
\end{array}
\]
\end{itemize}
\end{theorem}
\subsection{Nonlinear Integral Functionals.} 
\begin{lemma}\label{lem52} 
Let $E$ be a real normed space, $G$ be an open subset of $E$, $\mathfrak{f} : [0,T] \times G \rightarrow \R$ be a function, and $z_0 \in C^0([0,T], G)$. We consider the following conditions.
\begin{itemize}
\item[{\bf (i)}] $\mathfrak{f} \in C^0([0,T] \times G, \R)$.
\item[{\bf (ii)}] There exist $\rho > 0$ and $\zeta \in \mathcal{L}^1([0,T], \mathcal{B}([0,T]),\mathfrak{m}_1; \R_+)$ s.t., for all $t \in [0,T]$, for all $u_1,u_2 \in B(z_0(t), \rho)$, $\vert \mathfrak{f}(t, u_1) - \mathfrak{f}(t, u_2) \vert \leq \zeta (t) \Vert u_1 - u_2 \Vert$.
\item[{\bf (iii)}] For all $t \in [0,T]$, $D_{H,2} \mathfrak{f}(t, z_0(t))$ exists.
\end{itemize}
We consider the functional $F : C^0([0,T], G) \rightarrow \R$ defined by $F(z) := \int_0^T \mathfrak{f}(t, z(t))dt$ when $z \in C^0([0,T], G)$.\\
Then the following conclusions hold.
\begin{itemize}
\item[{\bf (a)}] Under (i-ii), $F$ is well defined and Lipschitzean on the ball $B_{\Vert \cdot \Vert_{\infty}}(z_0, \rho)$.
\item[{\bf (b)}] Under (i-iii), $F$ is Hadamard differentiable at $z_0$, and for all \\
$h \in C^0([0,T], E)$,  $[t \mapsto D_{H,2}\mathfrak{f}(t, z_0(t)) \cdot h(t)] \in \mathcal{L}^1([0,T], \mathcal{B}([0,T]), \mathfrak{m}_1; \R)$ 
and we have $D_H F(z_0) \cdot h = \int_{[0,T]} D_{H,2} \mathfrak{f}(t,z_0(t)) \cdot h(t) \; d \mathfrak{m}_1(t)$.
\end{itemize}
\end{lemma}
\begin{proof} {\bf (a)} Let $z \in C^0([0,T], G)$. Under (i), since $\mathfrak{f} \in C([0,T] \times G, \R)$, we have $[t \mapsto \mathfrak{f}(t, z(t))] \in C^0([0,T], \R)$; hence this function is Riemann integrable on $[0,T]$, and so $F(z)$ is well defined. Doing a straightforward majorization, we obtain the following inequality.
\begin{equation}\label{eq51}
\forall z,w \in B_{\Vert \cdot \Vert_{\infty}}(z_0, \rho), \hskip3mm  \vert F(z) - F(w) \vert \leq \Vert \zeta \Vert_{L^1} \cdot \Vert z - w \Vert_{\infty}.
\end{equation}
and so the conclusion (a) is proven.\\
{\bf (b)} When $z \in C^0([0,T], G)$, since $z([0,T])$ is compact and since the function $[u \mapsto d(u, E \setminus G) := \inf \{ \Vert u - v \Vert : v \in E \setminus G \}]$ is continuous (since Lipschitzean), using the Optimization Theorem of Weierstrass and the closedness of $E \setminus G$, we can assert that
$\alpha_z := \inf \{ d(u, E \setminus G) : u \in z([0,T]) \} > 0$. We can verify that $B_{\Vert \cdot \Vert_{\infty}}(z, 2^{-1} \alpha_z) \subset C^0([0,T], G)$.\\
Let $h \in C^0([0,T], E)$, $h \neq 0$ ( the case $h = 0$ is evident).\\
 We set $\theta^0 := \Vert h \Vert_{\infty}^{-1} \min \{ \rho, 2^{-1} \alpha_{z_0} \} > 0$. Hence, for all $\theta \in \; ]0, \theta^0[$, for all $t \in [0,T]$, we have $z_0(t) + \theta h(t) \in B(z_0(t), \rho)$. Let $(\theta_m)_{m \in \N} \in \; ]0,\theta^0[^{\N}$ which converges to $0$.  Using (iii), since the Hadamard differentiability implies the G\^ateaux differentiability, we have:
\begin{equation}\label{eq52}
\left.
\begin{array}{l}
D_{H,2} \mathfrak{f}(t,z_0(t)) \cdot h(t) = D^+_{G,2} \mathfrak{f}(t, z_0(t); h(t))\\
= \lim_{m \rightarrow + \infty} \frac{1}{\theta_m} (\mathfrak{f}(t, z_0(t) + \theta_m h(t)) - \mathfrak{f}(t, z_0(t))).
\end{array}
\right\}
\end{equation}
Since $[t \mapsto \mathfrak{f}(t, z_0(t) + \theta_m h(t)) ]$ and $[t \mapsto \mathfrak{f}(t, z_0(t))]$ belong to $C^0([0,T], \R)$, they are Borel functions and therefore, for all $m \in \N$, 
$[t \mapsto  \frac{1}{\theta_m} (\mathfrak{f}(t, z_0(t) + \theta_m h(t)) - \mathfrak{f}(t, z_0(t)))]$ is also a Borel function. Since a pointwise limit of a sequence of Borel functions is a Borel function, using (\ref{eq52}), we obtain:
\begin{equation}\label{eq53}
[t \mapsto D_{H,2} \mathfrak{f}(t,z_0(t)) \cdot h(t)] \in \mathcal{L}^0([0,T], \mathcal{B}([0,T]); \R).
\end{equation}
Using (ii), we obtain, for all $m \in \N$, 
\begin{equation}\label{eq54}
\left\vert \frac{1}{\theta_m} (\mathfrak{f}(t, z_0(t) + \theta_m h(t)) - \mathfrak{f}(t, z_0(t))) \right\vert \leq \zeta(t) \Vert h(t) \Vert \leq \zeta(t) \Vert h \Vert_{\infty}.
\end{equation}
Doing $m \rightarrow + \infty$, we deduce from (\ref{eq52}) and (\ref{eq54}) that 
$\vert D^+_{G,2} \mathfrak{f}(t,z_0(t); h(t)) \vert \leq \zeta(t) \Vert h \Vert_{\infty}$, and consequently, for all $t \in [0,T]$, we have
\begin{equation}\label{eq55} 
\vert D_{H,2} \mathfrak{f}(t, z_0(t)) \cdot h(t) \vert \leq \zeta(t) \Vert h \Vert_{\infty}.
\end{equation}
From (\ref{eq53}) and (\ref{eq55}), we obtain the following property.
\begin{equation}\label{eq56}
[t \mapsto D_{H,2} \mathfrak{f}(t,z_0(t)) \cdot h(t)] \in \mathcal{L}^1([0,T], \mathcal{B}([0,T]), \mathfrak{m}_1; \R).
\end{equation}
Note that, for all $m \in \N$, using (a) and the linearity of the Riemann integral, we have, for all $m \in \N$, $\frac{1}{\theta_m} (F(z_0 + \theta_m h) - F(z_0)) $ = \\
$\int_0^T \frac{1}{\theta_m} (\mathfrak{f}(t, z_0(t) + \theta_m h(t)) - \mathfrak{f}(t, z_0(t))) dt $ = \\\ 
$\int_{[0,T]}  \frac{1}{\theta_m} (\mathfrak{f}(t, z_0(t) + \theta_m h(t)) - \mathfrak{f}(t, z_0(t))) \; d \mathfrak{m}_1(t)$. Then, using (\ref{eq54}) and (\ref{eq52}), we can use the Dominated Convergence Theorem of Lebesgue to obtain that \\
$\lim_{n \rightarrow + \infty}\frac{1}{\theta_m} (F(z_0 + \theta_m h) - F(z_0)) = \int_{[0,T]} D_{H,2}\mathfrak{f}(t, z_0(t)) \cdot h(t) \; d\mathfrak{m}_1(t)$. Using the sequential characterization of the limit, we obtain the existence of $D^+_G F(z_0; h)$ and 
\begin{equation}\label{eq57}
D^+_G F(z_0; h) =  \int_{[0,T]}  D_{H,2} \mathfrak{f}(t, z_0(t)) \cdot h(t) \, d\mathfrak{m}_1(t).
\end{equation}
Using the linearity of the Borel integral and the linearity of the Hadamard differential at a point, we see that $D^+_G F(z_0; \cdot)$ is a linear functional from 
$C^0([0,T], E)$ into $\R$. Note that, using (\ref{eq57}) and (\ref{eq55}), we have:
$\vert D^+_G F(z_0; h) \vert  =  \vert \int_{[0,T]}  D_{H,2} \mathfrak{f}(t, z_0(t)) \cdot h(t) \; d\mathfrak{m}_1(t) \vert
\leq  \int_{[0,T]} \vert  D_{H,2} \mathfrak{f}(t, z_0(t)) \cdot h(t) \vert \; d \mathfrak{m}_1(t) \leq  \int_{[0,T]} ( \zeta(t) \Vert h \Vert_{\infty}) \; d \mathfrak{m}_1(t)$ \\
$ = \Vert \zeta \Vert_{L^1} \Vert h \Vert_{\infty}$,
and so $D^+_G F(z_0; \cdot)$ is linear continuous. Hence we have proven
\begin{equation}\label{eq58}
D_G F(z_0) \; {\rm exists}.
\end{equation}
Since $F$ is Lipschitzean, we can use (\cite{Fl}, p.259) to assert that $F$ is Hadamard differentiable at $z_0$, and the formula of this Hadamard differential is given by this one of its G\^ateaux differential.
\end{proof}
\begin{remark}\label{rem53}
Under the assumptions of Lemma \ref{lem52}, the following assertions hold.
\begin{itemize}
\item[{\bf (i)}] If $E$ is separable, then $[t \mapsto \Vert D_{H,2} \mathfrak{f}(t, z_0(t)) \Vert_*]$ is Borel integrable on $[0,T]$.
\item[{\bf (ii)}] If $E = \R^n$, then $[t \mapsto D_{H,2} \mathfrak{f}(t, z_0(t))]$ is Borel integrable on $[0,T]$ 
\end{itemize}
\end{remark}
\begin{proof} To abridge the writing, we set $\Lambda(t) := D_{H,2} \mathfrak{f}(t, z_0(t))$.\\
{\bf (i)} Since $E$ is separable (and metric), the closed unit ball of $E$, $\overline{B}_E(0,1)$, is also separable (\cite{Di}, (3.10.9), p.45), hence there exists $A \subset \overline{B}_E(0,1)$ which is at most countable and dense in $\overline{B}_E(0,1)$. Using (\ref{eq56}), for all $t \in [0,T]$, 
$\Vert \Lambda(t) \Vert_* = \sup \{ \vert \Lambda(t) \cdot v \vert : v \in \overline{B}_E(0,1) \} = \sup \{  \vert \Lambda(t) \cdot v \vert : v \in A \} \in  \mathcal{L}^0([0,T], \mathcal{B}([0,T]); \R)$ as a supremum of a sequence of functions which belong to $ \mathcal{L}^0([0,T], \mathcal{B}([0,T]); \R)^{\N}$. Using (\ref{eq55}) we obtain that, for all $t \in [0,T]$, $\Vert \Lambda(t) \Vert_* \leq \zeta(t)$, which implies that 
$\Vert \Lambda( \cdot) \Vert_* \in  \mathcal{L}^1([0,T], \mathcal{B}([0,T]), \mathfrak{m}_1; \R_+)$.\\
{\bf (ii)} Since $dim \R^{n*}$ is finite, the $\mathfrak{m}_1$-integrability of $\Lambda$ is equivalent to the $\mathfrak{m}_1$-integrability of its coordinate functions. Let $(e_i)_{1 \leq i \leq n}$ be the canonical basis of $\R^n$ and $(e^*_i)_{1 \leq i \leq n}$ its dual basis. Note that we have $\Lambda(t) = \sum_{1 \leq i \leq n}(\Lambda(t) \cdot e_i) e^*_i$. From (\ref{eq56}) we know that, for all $i \in \{1,...,n \}$, $\Lambda(\cdot) \cdot e_i$ is $\mathfrak{m}_1$-integrable on $[0,T]$, and consequently we obtain that $\Lambda \in \mathcal{L}^1([0,T], \mathcal{B}([0,T]), \mathfrak{m}_1; \R^{n*})$.
\end{proof}
\begin{remark}\label{rem54}
Consider the following strengthened condition:\\
{\bf (St)}: $\mathfrak{f} \in C^0([0,T] \times G, \R)$, $D_{F,2}\mathfrak{f}(t, \cdot)$ exists on $G$ for all $t \in [0,T]$, and $D_{F,2}\mathfrak{f} \in C^0([0,T] \times G, E^*)$.\\
Under (St), the Nemytskii operator $\mathcal{N}_{\mathfrak{f}} : C^0([0,T], G)) \rightarrow C^0([0,T], \R)$, defined by $\mathcal{N}_{\mathfrak{f}}(z)(t) := \mathfrak{f}(t, z(t))$ for all $t \in [0,T]$ and for all $z \in C^0([0,T], G)$, is Fr\'echet $C^1$ and $(D_F \mathcal{N}_{\mathfrak{f}}(z) \cdot h)(t) = D_{F,2} \mathfrak{f}(t, z(t)) \cdot h(t)$ for all $t \in [0,T]$ and for all $h \in C^0([0,T], E)$. This result is proven in \cite{BK} (Lemma 12). Since the Riemann integral defines a linear continuous functional $\mathcal{I}$ on $C^0([0,T], \R)$, the functional $F$, defined by $F(z) := \int_0^T \mathfrak{f}(t, z(t)) \; dt$, verifies $F = \mathcal{I} \circ \mathcal{N}_{\mathfrak{f}}$, and therefore $F$ is Fr\'echet $C^1$, and, using the Chain Rule, we have $D_F F(z) \cdot h = \int_0^T D_{F,2}\mathfrak{f}(t, z(t)) \cdot h(t) \; dt$. Under (St), the assumptions (i) and (iii) of Lemma \ref{lem52} are fulfilled. Using Lemma 12 of \cite{BK}, the Nemytskii operator $\mathcal{N}_{D_{F,2}{\mathfrak{f}}}$ is continuous, and since a mapping which is continuous at a point is bounded on a neighborhood of this point, and using the Mean Value Inequality we see that (ii) of Lemma \ref{lem52} is fulfilled. Hence Lemma \ref{lem52} contains a contribution to improve the results of \cite{BK} on the differentiability of the nonlinear integral functionals.
\end{remark}
\begin{lemma}\label{lem55}
let $M$ be an open subset of $\R^n$, ${\bf P}$ be an open subset of $Y$, $\phi : [0,T] \times M \times \R^n \times {\bf P} \rightarrow \R$ be a function, $\pi_0 \in {\bf P}$ and $[\pi \mapsto  {\bf x}(\pi)] $ be a mapping from ${\bf P}$ into $C^1([0,T], M)$. We consider the following conditions.
\begin{itemize}
\item[{\bf (i)}] $\phi \in C^0([0,T] \times M \times \R^n \times {\bf P}, \R)$.
\item[{\bf (ii)}] $[ \pi \mapsto {\bf x}(\pi)]$ is continuous at $\pi_0$.
\item[{\bf (iii)}] There exist $\varrho > 0$ and $\gamma \in \mathcal{L}^1([0,T], \mathcal{B}([0,T]), \mathfrak{m}_1; \R_+)$ s.t. $\forall t \in [0,T]$,$ \\\forall u,u_1 \in B({\bf x}(\pi_0)(t), \varrho)$, $\forall v, v_1 \in B({\bf x}(\pi_0)'(t), \varrho)$, $\forall \pi, \pi_1 \in B(\pi_0, \varrho)$, \\
$\vert \phi(t,u,v, \pi) - \phi(t,u_1,v_1, \pi_1) \vert \leq \gamma(t) (\Vert u-u_1 \Vert + \Vert v -v_1 \Vert + \Vert \pi - \pi_1 \Vert)$.
\item[{\bf (iv)}] For all $t \in [0,T]$, for all $\pi \in B(\pi_0, \varrho)$, $D_{H, (2,3,4)} \phi(t, {\bf x}(\pi)(t), {\bf x}(\pi)'(t), \pi)$ exists.
\item[{\bf (v)}] For all $t \in [0,T]$, $[ \pi \mapsto D_{H, (2,3)} \phi(t,{\bf x}(\pi)(t),{\bf x}(\pi)'(t), \pi)]$ is continuous on $B(\pi_0, \varrho)$.
\item[{\bf (vi)}] For all $t \in [0,T]$, $[ \pi \mapsto D_{H, (2,3,4)} \phi(t, {\bf x}(\pi)(t), {\bf x}(\pi)'(t), \pi)]$ is continuous on $B(\pi_0, \varrho)$.
\item[{\bf (vii)}] For all $\pi \in B(\pi_0, \varrho)$, $[t \mapsto D_{H,(2,3,4)}\phi(t, {\bf x}(\pi)(t), {\bf x}(\pi)'(t), \pi)]$ belongs to $\mathcal{L}^0([0,T], \mathcal{B}([0,T]); (\R^n \times \R^n \times Y)^*)$.
\end{itemize}
We consider the functional $\Phi : C^1([0,T], M) \times {\bf P} \rightarrow \R$ defined by $\Phi({\bf x},\pi) := \int_0^T \phi(t, {\bf x}(t), {\bf x}'(t), \pi) \; dt$.
Then the following conclusions hold.
\begin{itemize}
\item[{\bf (a)}] Under (i-iii), there exists $\sigma \in \;]0, \varrho]$ s.t. for all $\pi \in B(\pi_0, \sigma)$, $[{\bf x} \mapsto \Phi({\bf x}, \pi)]$  
is Lipschitzean on $B_{C^1}({\bf x}(\pi), \sigma)$.
\item[{\bf (b)}] Under (i-iv), for all $\pi \in B(\pi_0, \sigma)$, $D_H \Phi({\bf x}(\pi), \pi)$ exists, and for all ${\bf h} \in C^1([0,T], \R^n)$, for all $\varpi \in Y$, \\
$[t \mapsto D_{H, (2,3,4)} \phi(t, {\bf x}(\pi)(t), {\bf x}(\pi)'(t), \pi) \cdot ( {\bf h}(t), {\bf h}'(t), \varpi)]$ belongs to\\ $\mathcal{L}^1([0,T], \mathcal{B}([0,T]), \mathfrak{m}_1; \R)$ and 
\[
\begin{array}{ccl}
D_H \Phi({\bf x}(\pi), \pi) \cdot ({\bf h}, \varpi) & = & \int_{[0,T]} D_{H,2} \phi(t, {\bf x}(\pi)(t), {\bf x}(\pi)'(t), \pi) \cdot {\bf h}(t) \; d \mathfrak{m}_1(t)\\
\null & \null & + \int_{[0,T]} D_{H,3} \phi(t, {\bf x}(\pi)(t), {\bf x}(\pi)'(t), \pi) \cdot {\bf h}'(t) \; d \mathfrak{m}_1(t)\\
\null & \null & + \int_{[0,T]} D_{H,4} \phi(t, {\bf x}(\pi)(t), {\bf x}(\pi)'(t), \pi) \cdot \varpi \; d \mathfrak{m}_1(t).
\end{array}
\]
\item[{\bf (c)}] Under (i-v), $[ \pi \mapsto  D_{H,1} \Phi({\bf x}(\pi), \pi)]$ is continuous from $B(\pi_0, \sigma)$ into\\ $(C^1([0,T], \R^n))^*$.
\item[{\bf (d)}] Under (i-vii), $[ \pi \mapsto  D_{H} \Phi({\bf x}(\pi), \pi)]$ is continuous from $B(\pi_0, \sigma)$ into\\ $(C^1([0,T], \R^n) \times Y)^*$.
\end{itemize}
\end{lemma}
\begin{proof} Let $\varrho$ be given by (iii). Using (ii), there exists $\sigma \in \;] 0,\frac{\varrho}{2}]$ s.t. $\Vert {\bf x}(\pi) - {\bf x}(\pi_0) \Vert_{C^1} \leq \frac{\varrho}{2}$ when $\Vert \pi - \pi_0 \Vert < \sigma$. Using (iii), we obtain the following property.
\begin{equation}\label{eq59}
\left.
\begin{array}{l}
\forall \pi \in B(\pi_0,\sigma), \forall t \in [0,T], \forall u_1, u_2 \in B({\bf x}(\pi)(t), \sigma),
\forall v_1, v_2 \in B({\bf x}(\pi)'(t), \sigma),\\
 \forall \pi_1, \pi_2 \in B(\pi, \sigma),
\vert \phi(t,u_1,v_1,\pi_1) - \phi(t, u_2,v_2,\pi_2) \vert \leq \\
\gamma(t) (\Vert u_1 - u_2 \Vert + \Vert v_1 - v_2 \Vert + \Vert \pi_1 - \pi_2 \Vert).
\end{array}
\right\}
\end{equation}
We want to use Lemma \ref{lem52}. We set $E := \R^n \times \R^n \times Y$ and $G := M \times \R^n \times {\bf P}$ which is an open subset of $E$. At each $\pi \in {\bf P}$ we associate the constant mapping $\pi_c := [ t \mapsto \pi] \in C^0([0,T], {\bf P})$. We define the function ${\mathfrak f} : [0,T] \times G \rightarrow \R$ by setting ${\mathfrak f}(t,( u,v,\pi)) := \phi(t,u,v,\pi)$. When ${\bf x} \in C^1([0,T], M)$ and $\pi \in {\bf P}$, setting $z(t) := ({\bf x}(t), {\bf x}'(t), \pi_c(t))$, we have $z \in C^0([0,T], G)$, and $F(z) = \int_0^T {\mathfrak f}(t,z(t)) \; dt = \int _0^T \phi(t, {\bf x}(t), {\bf x}'(t), \pi) \; dt = \Phi({\bf x},\pi)$. Hence $\Phi$ can be viewed as a restriction of $F$ to an open subset of a closed (since complete) vector subspace.\\
{\bf Proof of (a).} Let $\pi \in B(\pi_0, \sigma)$; we set $z_0(t) := ({\bf x}(\pi)(t), {\bf x}(\pi)'(t), (\pi)_c(t))$. Now we verify that that the assumptions of Lemma \ref{lem52} are fulfilled. Note that (i) implies that the assumption (i) of Lemma \ref{lem52} is fulfilled. After (\ref{eq59}), the assumption (ii) of Lemma \ref{lem52} is fulfilled. Hence we can use the conclusion (a) of Lemma \ref{lem52} to ensure that $[{\bf x} \mapsto \Phi({\bf x}, \pi)]$ is Lipschitzean on $B_{C^1}({\bf x}(\pi), \sigma)$, and so the conclusion (a) is proven.\\
{\bf Proof of (b).} We arbitrarily fix $\pi \in B(\pi_0, \sigma)$, and we set \\
$z_0(t) := ({\bf x}(\pi)(t), {\bf x}(\pi)'(t), \pi_c(t))$. We can verify that assumptions of Lemma \ref{lem52} are fulfilled for $z_0$. In the proof of (a) we have yet proved that assumptions (i) and (ii) of Lemma \ref{lem52} are fulfilled. Moreover assumption (iv) of Lemma \ref{lem55} implies that the assumption (iii) of Lemma \ref{lem52} is fulfilled. Hence we can use the conclusion (b) of Lemma \ref{lem52} and assert that $D_H \Phi({\bf x}(\pi), \pi) = D_HF(z_0)$ exists. We introduce the operator 
$\Psi : C^1([0,T], \R^n) \times Y \rightarrow C^0([0,T], \R^n) \times C^0([0,T], \R^n) \times Y$ defined by $\Psi({\bf x},\pi) := ({\bf x},{\bf x}',\pi)$. $\Psi$ is linear continuous, hence it is Fr\'echet differentiable, and consequently Hadamard differentiable. We set $\Psi_0$ the restriction of $\Psi$ to $C^1([0,T], M) \times {\bf P}$ and we note that $\Phi = F \circ \Psi_0$. Using Lemma \ref{lem52}, we know that $[t \mapsto D_{H,2} \mathfrak{f}(t, z_0(t)) \cdot w(t)]$ belongs  to $\mathcal{L}^1([0,T], \mathcal{B}([0,T]), \mathfrak{m}_1; \R)$ for all $w \in C^0([0,T], E)$, and $D_HF(z_0) \cdot w = \int_{[0,T]}D_{H,2} \mathfrak{f}(t,z_0(t)) \cdot w(t) \; d \mathfrak{m}_1(t)$. Note that $D_{H,2} \mathfrak{f}(t, z_0(t)) = D_{H, (2,3,4)} \phi(t, {\bf x}(\pi)(t), {\bf x}(\pi)'(t), \pi)$ and $D_H \Psi_0({\bf x}(\pi), \pi) \cdot ({\bf h}, \varpi) = \Psi({\bf h}, \varpi) = ({\bf h}, {\bf h}', \varpi)$, and using the Chain Rule, we obtain the following formula.
\begin{equation}\label{eq510}
\left.
\begin{array}{l}
D_H \Phi({\bf x}(\pi), \pi) \cdot ({\bf h}, \varpi) = \\
\int_{[0,T]} D_{H, (2,3,4)} \phi(t, {\bf x}(\pi)(t), {\bf x}(\pi)'(t), \pi) \cdot ({\bf h}(t), {\bf h}'(t), \varpi) \; d \mathfrak{m}_1(t).
\end{array}
\right\}
\end{equation}
From this last relation we deduce the formula of the conclusion (b).\\
{\bf Proof of (c).} From (\ref{eq510}) we deduce the following formula.
\begin{equation}\label{eq511}
\left.
\begin{array}{l}
D_{H,1} \Phi({\bf x}(\pi), \pi) \cdot {\bf h} = \\
\int_{[0,T]} D_{H, (2,3)} \phi(t, {\bf x}(\pi)(t), {\bf x}(\pi)'(t), \pi) \cdot ({\bf h}(t), {\bf h}'(t)) \; d \mathfrak{m}_1(t).
\end{array}
\right\}
\end{equation}
Let $\pi \in B(\pi_0, \sigma)$ and $(\pi_k)_{k \in \N} \in B(\pi_0, \sigma)^{\N}$ which converges to $\pi$. When $k \in \N$, and $t \in [0,T]$, we set 
$$\Gamma_k(t) := \Vert  D_{H, (2,3)} \phi(t,{\bf x}(\pi_k)(t), {\bf x}(\pi_k)'(t), \pi_k) -  D_{H, (2,3)} \phi(t, {\bf x}(\pi)(t), {\bf x}(\pi)'(t), \pi) \Vert_*$$
where the norm is the norm of $(\R^n \times \R^n)^*$. From (\ref{eq59}) and (v), we obtain the following properties.
\begin{equation}\label{eq512}
\forall t \in [0,T], \; \Gamma_k(t) \leq 2 \gamma(t), \; {\rm and} \; \lim_{k \rightarrow + \infty} \Gamma_k(t) = 0.
\end{equation}
Let $p \in \{ \pi \} \cup \{ \pi_k : k \in \N \}$. Since, for all ${\bf h}, {\bf k} \in C^0([0,T], \R^n)$, \\
$[t \mapsto D_{H, (2,3)} \phi(t, {\bf x}(p)(t), {\bf x}(p)'(t), p) \cdot ({\bf h}(t),{\bf k}(t))] \in \mathcal{L}^1([0,T], \mathcal{B}([0,T]), \mathfrak{m}_1; \R)$, we have that $[t \mapsto D_{H, (2,3)} \phi(t, {\bf x}(p)(t), {\bf x}(p)'(t), p) \cdot ({\bf h}(t),{\bf k}(t))] \in \mathcal{L}^0([0,T], \mathcal{B}([0,T]); \R)$. Hence, for all $v,w \in \R^n$,\\
 $[t \mapsto D_{H, (2,3)} \phi(t, {\bf x}(p)(t), {\bf x}(p)'(t), p) \cdot (v,w)] \in \mathcal{L}^0([0,T], \mathcal{B}([0,T]); \R)$. Let $(e_i)_{1 \leq i \leq 2n}$ be the canonical basis of $\R^n \times \R^n$, and $(e_i^*)_{1 \leq i \leq 2n}$ denotes its dual basis. Note that, for all $t \in [0,T]$, we have
$$D_{H, (2,3)} \phi(t, {\bf x}(p)(t), {\bf x}(p)'(t), p) = \sum_{1 \leq i \leq 2n}( D_{H, (2,3)} \phi(t, {\bf x}(p)(t), {\bf x}(p)'(t), p) \cdot e_i) e_i^*,$$
hence, as a composition of Borel functions, we obtain
$$
[t \mapsto D_{H, (2,3)} \phi(t, {\bf x}(p)(t), {\bf x}(p)'(t), p)]
\in \mathcal{L}^0([0,T], \mathcal{B}([0,T]); (\R^n \times \R^n)^*).
$$
As compositions of Borel functions, we deduce of this property that 
\begin{equation}\label{eq513}
\forall k \in \N, [t \mapsto \Gamma_k(t)] \in  \mathcal{L}^0([0,T], \mathcal{B}([0,T]); \R).
\end{equation}
Since $\gamma$ is $\mathfrak{m}_1$-integrable on $[0,T]$, from (\ref{eq513}) and (\ref{eq512}), we obtain that $\Gamma_k$ is $\mathfrak{m}_1$-integrable on $[0,T]$ for all $k \in \N$. Hence we can do the following majorizations. For all $k \in \N$, for all ${\bf h} \in C^1([0,T], \R^n)$ s.t. $\Vert {\bf h} \Vert_{C^1} \leq 1$, we have 
\[ 
\begin{array}{l}
\vert D_{H,1} \Phi({\bf x}(\pi_k), \pi_k) \cdot h - D_{H,1} \Phi({\bf x}(\pi)), \pi) \cdot {\bf h} \vert\\
= \vert \int_{[0,T]} ( D_{H,(2,3)} \phi(t, {\bf x}(\pi_k)(t), {\bf x}(\pi_k)'(t), \pi_k) \\
- D_{H, (2,3)} \phi(t, {\bf x}(\pi)(t), {\bf x}(\pi)'(t), \pi))\cdot( {\bf h}(t), {\bf h}'(t)) \; d \mathfrak{m}_1(t) \vert\\
\leq \int_{[0,T]} (\Gamma_k(t) \Vert ({\bf h}(t), {\bf h}'(t)) \Vert\; d \mathfrak{m}_1(t)
\leq ( \int_{[0,T]} \Gamma_k(t)\; d \mathfrak{m}_1(t)) \Vert {\bf h} \Vert_{C^1}.
\end{array}
\]
Taking the supremum on the ${\bf h} \in C^1([0,T], \R^n)$ s.t. $\Vert {\bf h} \Vert_{C^1} \leq 1$, we obtain 
\begin{equation}\label{eq514}
\Vert D_{H,1} \Phi({\bf x}(\pi_k), \pi_k)  - D_{H,1} \Phi({\bf x}(\pi)), \pi) \Vert_{(C^1([0,T], \R^n))^*} \leq \int_{[0,T]} \Gamma_k(t) \; d \mathfrak{m}_1(t).
\end{equation}
Using (\ref{eq514}) and (\ref{eq512}) we can apply the dominated convergence theorem of Lebesgue to obtain that\\
$\lim_{k \rightarrow + \infty} \Vert D_{H,1} \Phi({\bf x}(\pi_k), \pi_k)  - D_{H,1} \Phi({\bf x}(\pi)), \pi) \Vert_{(C^1([0,T], \R^n))^*} = 0$, and using the sequential characterization of the continuity, we have proven the conclusion (c).\\
{\bf Proof of (d).} Let $\pi \in B(\pi_0, \sigma)$ and $(\pi_k)_{k \in \N} \in B(\pi_0, \sigma)^{\N}$ which converges to $\pi$. When $k \in \N$ and $t \in [0,T]$, we set
$$\Delta_k(t) := \Vert D_{H,(2,3,4)} \phi(t,{\bf x}(\pi_k)(t), {\bf x}(\pi_k)'(t), \pi_k) - D_{H, (2,3,4)} \phi(t, {\bf x}(\pi)(t), {\bf x}(\pi)'(t), \pi) \Vert_*,$$
where the norm is the norm of $(\R^n \times \R^n \times Y)^*$. Using (vii), as compositions of Borel functions we have
\begin{equation}\label{eq515}
\forall k \in \N, [t \mapsto \Delta_k(t)] \in \mathcal{L}^0([0,T], \mathcal{B}([0,T]); \R_+).
\end{equation}
Proceeding as in the proof of (c) to establish (\ref{eq514}), we obtain, for all  $\varpi \in Y$ and for all ${\bf h} \in C^1([0,T], \R^n)$ s.t. $\Vert {\bf h} \Vert_{C^1} + \Vert \varpi \Vert \leq 1$,\\
 $\vert D_{H} \Phi({\bf x}(\pi_k), \pi_k) \cdot ({\bf h}, \varpi)  - D_{H} \Phi({\bf x}(\pi)), \pi) \cdot({\bf h}, \varpi) \vert 
\leq (\int_{[0,T]} \Delta_k(t) \; d \mathfrak{m}_1(t)) (\Vert {\bf h} \Vert_{C^1} + \Vert \varpi \Vert )$, and taking the l.u.b. on the $({\bf h}, \varpi)$ s.t. $\Vert {\bf h} \Vert_{C^1} + \Vert \varpi \Vert \leq 1$, we obtain, for all $k \in \N$,
\begin{equation}\label{eq516}
\Vert D_H \Phi({\bf x}(\pi_k), \pi_k) - D_H \Phi({\bf x}(\pi), \pi) \Vert_{(C^1([0,T], \R^n) \times Y)^*} \\
\leq \int_{[0,T]} \Delta_k(t) \; d \mathfrak{m}_1(t).
\end{equation}
From assumption (vi), we deduce that $\lim_{k \rightarrow + \infty} \Delta_k(t) = 0$ for all $t \in [0,T]$. Proceeding as in the proof of (c), we obtain
$0 \leq \Delta_k(t) \leq 2 \gamma(t)$ for all $k \in \N$ and for all $t \in [0,T]$. Then we can use the dominated convergence theorem of Lebesgue to obtain $\lim_{k \rightarrow + \infty} \int_{[0,T]} \Delta_k(t) \; d \mathfrak{m}_1(t) = 0$. From (\ref{eq514}), we deduce that $\lim_{k \rightarrow + \infty} \Vert D_H \Phi({\bf x}(\pi_k), \pi_k) - D_H \Phi({\bf x}(\pi), \pi) \Vert_{(C^1([0,T], \R^n) \times Y)^*} = 0$, and using the sequential characterization of the continuity, (d) is proven.
\end{proof}
\begin{remark}\label{rem56} Working as in the proof of Remark \ref{rem53}, if $Y$ is separable, under the assumption (vii) we obtain that $[t \mapsto \Delta_k(t)]$ is a Borel function without to use assumption (vii).
\end{remark}
\begin{remark}\label{rem57}  We fix $\pi \in B(\pi_0, \sigma)$. The property on the integrability of $[t \mapsto D_{H,(2,3,4)}\phi (t, {\bf x}(\pi)(t), {\bf x}(\pi)'(t), \pi) \cdot ({\bf h}(t), {\bf h}'(t), \varpi)]$ given in the conclusion (b) of Lemma \ref{lem55} implies the two following properties.
$$ (b1) \hskip7mm [t \mapsto D_{H,2} \phi (t, {\bf x}(\pi)(t), {\bf x}(\pi)'(t), \pi) ] \in \mathcal{L}^1([0,T], \mathcal{B}([0,T]), \R^{n*}; \mathfrak{m}_1).$$
$$ (b2) \hskip7mm [t \mapsto D_{H,3} \phi (t, {\bf x}(\pi)(t), {\bf x}(\pi)'(t), \pi) ] \in \mathcal{L}^1([0,T], \mathcal{B}([0,T]), \R^{n*}; \mathfrak{m}_1).$$
To prove them we consider $(e_i)_{1 \leq i \leq n}$ the canonical basis of $\R^n$; we denote by ${\bf e}_i$  the constant function on $[0,T]$ equal to $e_i $. We have ${\bf e}_i \in C^1([0,T], \R^n)$ and ${\bf e}_i' = 0$. Hence the function $[t \mapsto D_{H,2} \phi(t,{\bf x}(\pi)(t), {\bf x}(\pi)'(t), \pi) \cdot e_i =$ \\ 
$D_{H,(2,3,4)}\phi (t, {\bf x}(\pi)(t), {\bf x}(\pi)'(t), \pi) \cdot ({\bf e}_i(t), 0, 0)]$ is $\mathfrak{m}_1$-integrable on $[0,T]$ for all $i \in \{1, ..., n \}$, and so (b1) is proven.\\
Now we consider the function ${\bf a}_i := [ t \mapsto t e_i] \in C^1([0,T], \R^n)$ for all $i \in \{ 1,...,n \}$. Hence the function $[t \mapsto D_{H,2} \phi (t, {\bf x}(\pi)(t), {\bf x}(\pi)'(t), \pi) \cdot {\bf a}_i(t) +$ \\
$ D_{H,3} \phi (t, {\bf x}(\pi)(t), {\bf x}(\pi)'(t), \pi) \cdot e_i) = D_{H,(2,3,4)}\phi (t, {\bf x}(\pi)(t), {\bf x}(\pi)'(t), \pi)  \cdot ( {\bf a}_i(t), {\bf a}_i'(t), 0)]$  is $\mathfrak{m}_1$-integrable on $[0,T]$ for all $i \in \{1, ..., n \}$. \\
Note that $[t \mapsto D_{H,2} \phi (t, {\bf x}(\pi)(t), {\bf x}(\pi)'(t), \pi) \cdot (t e_i)]$ is a Borel function, and we have, for all $t \in [0,T]$, \\
 $\vert D_{H,2} \phi (t, {\bf x}(\pi)(t), {\bf x}(\pi)'(t), \pi) \cdot (t e_i) \vert \leq \Vert D_{H,2} \phi (t, {\bf x}(\pi)(t),{\bf x}(\pi)'(t), \pi) \Vert T \Vert e_i \Vert$ which is $\mathfrak{m}_1$-integrable on $[0,T]$. Therefore $[t \mapsto D_{H,2} \phi (t, {\bf x}(\pi)(t), {\bf x}(\pi)'(t), \pi) \cdot (t e_i)]$ is  $\mathfrak{m}_1$-integrable on $[0,T]$. Since $D_{H,3} \phi (t, {\bf x}(\pi)(t), {\bf x}(\pi)'(t), \pi) \cdot e_i =$ \\
 $ [ D_{H,2} \phi (t, {\bf x}(\pi)(t), {\bf x}(\pi)'(t), \pi) \cdot (t e_i)$ \\
 $ + D_{H,3} \phi (t, {\bf x}(\pi)(t), {\bf x}(\pi)'(t), \pi) \cdot e_i -  D_{H,2} \phi (t, {\bf x}(\pi)(t), {\bf x}(\pi)'(t), \pi) \cdot (t e_i)]$, we obtain that $[t \mapsto D_{H,3} \phi (t, {\bf x}(\pi)(t), {\bf x}(\pi)'(t), \pi) \cdot e_i]$  is $\mathfrak{m}_1$-integrable on $[0,T]$ as a difference of $\mathfrak{m}_1$-integrable functions for all $i \in \{1,...,n \}$; and so (b2) is proven.
\end{remark}

\subsection{Euler equation}
\begin{lemma}\label{lem58} In the setting of Lemma \ref{lem55}, let $\pi \in B(\pi_0, \alpha)$. Under conditions (i-iv) of Lemma \ref{lem55}, The two following assertions are equivalent. 
\begin{itemize}
\item[(i)] $\forall {\bf  h} \in C^1_{0,0}([0,T], \R^n)$, \;\; $D_{H,1} \Phi({\bf x}(\pi), \pi) \cdot {\bf h} = 0$.
\item[(ii)] There exists $c \in \R^{n*}$ s.t. $\mathfrak{m}_1$-a.e. $t \in [0,T]$,\\
$D_{H,3} \phi(t, {\bf x}(\pi)(t), {\bf x}(\pi)'(t), \pi) = \int_{[0,t]} D_{H,2} \phi(s, {\bf x}(\pi)(s), {\bf x}(\pi)'(s), \pi)\; d \mathfrak{m}_1(s) + c.$
\end{itemize}
\end{lemma}
\begin{proof}  Setting $\mathfrak{M}(t) := D_{H,2}\phi(t, {\bf x}(\pi)(t), {\bf x}(\pi)'(t), \pi)$  and\\
 $\mathfrak{N}(t) := D_{H,3}\phi(t, {\bf x}(\pi)(t), {\bf x}(\pi)'(t), \pi)$ when $t \in [0,T]$, using Lemma \ref{lem55}, we obtain that, for all ${\bf h} \in C^{\infty}_c([0,T], \R^n)$, 
\begin{equation}\label{eq517}
D_{H,1} \Phi({\bf x}(\pi), \pi) \cdot h = \int_{[0,T]}(\mathfrak{M}(t) \cdot{\bf h}(t) + \mathfrak{N}(t) \cdot {\bf h}'(t)) \; d \mathfrak{m}_1(t).
\end{equation}
From Remark \ref{rem57}, we know that $\mathfrak{M}, \mathfrak{N} \in \mathcal{L}^1([0,T], \mathcal{B}([0,T]), \mathfrak{m}_1; \R^{n*})$. We define $\mathfrak{P}(t) := \int_{[0,t]} \mathfrak{M}(s) \; d \mathfrak{m}_1(s)$; we have $\mathfrak{P} \in AC([0,T], \R^{n*})$ and $\mathfrak{P}'(t) = \mathfrak{M}(t)$ \; $\mathfrak{m}_1$-a.e. $t \in [0,T]$. When ${\bf h} \in C^{\infty}_c([0,T], \R^n)$, $\mathfrak{P} \cdot {\bf h} \in AC([0,T], \R)$ and the formula of integration by parts (\cite{H}, Annexe) holds: $\int_{[0,T]} \mathfrak{P}'(t) \cdot {\bf h}(t) \; d \mathfrak{m}_1(t) = \mathfrak{P}(T) \cdot {\bf h}(T) - \mathfrak{P}(0) \cdot {\bf h}(0) - \int_{[0,T]} \mathfrak{P}(t) \cdot {\bf h}'(t) \; d \mathfrak{m}_1(t)$, and since ${\bf h}(T) = {\bf h}(0) = 0$, we have
\begin{equation}\label{eq518}
\int_{[0,T]} \mathfrak{P}'(t) \cdot {\bf h}(t) \; d \mathfrak{m}_1(t) =  - \int_{[0,T]} 
 \mathfrak{P}(t) \cdot {\bf h}'(t) \; d \mathfrak{m}_1(t).
\end{equation} 
\vskip1mm
\noindent
$[{\bf (i) \Longrightarrow (ii)}]$ From (i), (\ref{eq517}) and (\ref{eq518}), for all ${\bf h} \in C^{\infty}_c([0,T], \R^n)$, we have
\[
\begin{array}{ccl}
0 & = &  \int_{[0,T]}(\mathfrak{M}(t) \cdot {\bf h}(t) + \mathfrak{N}(t) \cdot {\bf h}'(t)) \; d \mathfrak{m}_1(t)\\
\null & = &  \int_{[0,T]}(\mathfrak{P}'(t) \cdot {\bf h}(t) + \mathfrak{N}(t) \cdot {\bf h}'(t)) \; d \mathfrak{m}_1(t)\\
\null & = &  \int_{[0,T]}(- \mathfrak{P}(t) \cdot {\bf h}'(t) + \mathfrak{N}(t) \cdot {\bf h}'(t)) \; d \mathfrak{m}_1(t)\\
\null & = & \int_{[0,T]}(\mathfrak{N}(t) - \mathfrak{P}(t)) \cdot {\bf h}'(t) \; d \mathfrak{m}_1(t).
\end{array}
\]
Hence using the DuBois-Reymond lemma (\cite{BGH}, Lemma 1.8, p.15) we obtain that there exists  $c \in \R^{n*}$ s.t. $\mathfrak{N}(t) = \mathfrak{P}(t) +c$ \; $\mathfrak{m}_1$-a.e. $t \in [0,T]$ which is (ii).
\vskip1mm
\noindent
$[{\bf (ii) \Longrightarrow (i)}]$ For all ${\bf h} \in C^1_{0,0}([0,T], \R^n)$, note that $\int_{[0,T]} c \cdot {\bf h}'(t) \; d \mathfrak{m}_1(t)=\int_{0,T]} (c \cdot {\bf h})'(t) \; d \mathfrak{m}_1(t) = c \cdot {\bf h}(T) - c \cdot {\bf h}(0) = 0$. Using (ii) and (\ref{eq517}), we have $D_{H,1} \Phi({\bf x}(\pi), \pi) \cdot {\bf h} = \int_{[0,T]} (\mathfrak{M}(t) \cdot {\bf h}(t) + \mathfrak{N}(t) \cdot {\bf h}'(t)) \; d \mathfrak{m}_1(t) = \int_{[0,T]}(\mathfrak{P}'(t) \cdot {\bf h}(t) + \mathfrak{P}(t) \cdot {\bf h}'(t)) \; d \mathfrak{m}_1(t) + \int_{[0,T]} c \cdot {\bf h}'(t) \; d \mathfrak{m}_1(t)$, and using (\ref{eq518}), we obtain $D_{H,1} \Phi(x(\pi), \pi) \cdot {\bf h} = 0$ which is (i).
\end{proof}
\begin{lemma}\label{lem59}
Under (B{\sc sol}1), (B{\sc int}1), (B{\sc int}2) and (B{\sc int}3), the condition (B{\sc con}) is equivalent to the linear independence of $D_{H,1}G_1(x(\pi_0), \pi_0)$, ..., $D_{H,1}G_k(x(\pi_0), \pi_0)$, $D_{H,1}H_1(x(\pi_0), \pi_0)$, ..., $D_{H,1}H_{\ell}(x(\pi_0), \pi_0)$ on $C^1_{0,0}([0,T], \R^n)$.
\end{lemma}
\begin{proof} We want to use Lemma \ref{lem58} with $\phi = \psi_{\lambda, \mu}$, where $\psi_{\lambda, \mu}$ is defined in (B{\sc con}).
About the assumptions (i-iv) of Lemma \ref{lem55}, note that (i) is a consequence of (B{\sc int}2), (ii) is a consequence of (B{\sc sol}1), (iii) is a consequence of (B{\sc int}1), (iv) is  a consequence of (B{\sc int}2). Hence from Lemma \ref{lem55}, we know that $\Phi : C^1([0,T],M) \times P \rightarrow \R$, defined by $\Phi({\bf x}, \pi) := \int_0^T \phi(t, {\bf x}(t), {\bf x}'(t), \pi) \, dt = \int_0^T  \psi_{\lambda, \mu}(t, {\bf x}(t), {\bf x}'(t), \pi) \, dt$, is Hadamard differentiable at $x(\pi_0)$. To realize a proof, we proceeed doing a double contraposition. The negation of (B6) is equivalent to: $\exists (\lambda, \mu) \in \R^k \times \R^{\ell} \setminus \{(0,0) \}$, $\exists c_{\lambda, \mu} \in \R^{n*}$ s.t.\\
 $D_{H,3} \psi_{\lambda,\mu}(t, x(\pi_0)(t), x(\pi_0)'(t),\pi_0) =$ \\
 $ \int_{[0,t]} D_{H,2} \psi_{\lambda, \mu}(s, x(\pi_0)(s), x(\pi_0)'(s), \pi_0) \; d \mathfrak{m}_1(s) + c_{\lambda, \mu}$ \; $\mathfrak{m}_1$-a.e. $t \in [0,T]$. Using Lemma \ref{lem58}, this last assertion is equivalent to: $\exists (\lambda, \mu) \in \R^k \times \R^{\ell} \setminus \{(0,0) \}$ s.t. $D_{H,1}\Phi (x(\pi_0), \pi_0) = 0$ on $C^1_{0,0}([0,T], \R^n)$, i.e. \\$\sum_{1 \leq i \leq k} \lambda_i D_{H,1} G_i(x(\pi_0), \pi_0) + \sum_{1 \leq j \leq {\ell}} \mu_j D_{H,1} H_j(x(\pi_0), \pi_0) = 0$ on $C^1_{0,0}([0,T], \R^n)$. This last assertion means the linear dependence of $D_{H,1} G_1(x(\pi_0), \pi_0)$, ...,$\\
 $ $D_{H,1} H_{\ell}(x(\pi_0), \pi_0)$ on $C^1_{0,0}([0,T], \R^n)$. 
\end{proof}
\subsection{The dual space of $C^1_{0,0}([0,T], \R^n)$}
\begin{lemma}\label{lem510}
Let $V$ and $W$ be two real normed spaces. When $(v^*, w^*) \in V^* \times W^*$, we consider the {\sf direct sum} of $v^*$ and $w^*$, $v^* \oplus w^* \in (V \times W)^*$, defined by $v^* \oplus w^*(v,w) := v^*(v) + w^*(w)$ for all $(v,w) \in V \times W$. We define the operator $\mathcal{S} : V^* \times W^* \rightarrow (V \times W)^*$ by setting $\mathcal{S}(v^*, w^*) := v^* \oplus w^*$ when $(v^*, w^*) \in V^* \times W^*$.\\
Then $\mathcal{S}$ is a topological isomorphism from $V^* \times W^*$ onto $(V \times W)^*$.
\end{lemma}
This result is given in \cite{ATF} (Lemme 2, p.114) where its proof is left as an exercice since it is very easy.
\noindent
We consider $Af([0,T], \R^n)$, the set of the restrictions to $[0,T]$ of the affine functions from $\R$ into $\R^n$. Note that $\alpha \in Af([0,T], \R^n)$ means that there exists $(\eta, \sigma) \in \R^n \times \R^n$ s.t. $\alpha(t) = t \eta + \sigma$ for all $t \in [0,T]$. Clearly $Af([0,T], \R^n)$ is a vector subspace of $C^1([0,T], \R^n)$.
\begin{lemma}\label{lem511}
$C^1([0,T], \R^n) = Af([0,T], \R^n) \oplus C^1_{0,0}([0,T], \R^n)$ (topological direct sum).
\end{lemma}
\begin{proof} When $\varphi \in Af([0,T], \R^n) \cap  C^1_{0,0}([0,T], \R^n)$, we have $\varphi (t) = t \eta + \sigma$ for all $t \in [0,T]$. Since $\varphi \in C^1_{0,0}([0,T], \R^n)$ we have $0 = \varphi(0) = \sigma$ and $0 = \varphi(T) = T \eta + \sigma$ which imply $\sigma = \eta = 0$, and consequently $\varphi = 0$. Hence we have established the following property on the algebraic direct sum : $ Af([0,T], \R^n) \oplus^a C^1_{0,0}([0,T], \R^n) \subset C^1([0,T], \R^n)$. When $\varphi \in C^1([0,T], \R^n)$, we set $\eta := - \frac{1}{T}( \varphi(0) - \varphi(T))$, $\sigma := \varphi(0)$ and $\alpha(t) := t \eta + \sigma$ for all $t \in [0,T]$. Hence we have $\alpha \in Af([0,T], \R^n)$. We define $\psi : [0,T] \rightarrow \R^n$ by setting $\psi(t) := \varphi(t) + \frac{t}{T}(\varphi(0) - \varphi(T)) - \varphi(0)$ for all $t \in [0,T]$. Note that $\psi \in C^1([0,T], \R^n)$, $\psi(0) = \varphi(0) - \varphi(0) = 0$, and $\psi(T) = \varphi(T) + (\varphi(0) - \varphi(T)) - \varphi(0) = 0$. Hence we have $\psi \in C^1_{0,0}([0,T], \R^n)$. We see that $\alpha(t) + \psi(t) = \varphi(t)$ for all $t \in [0,T]$, i.e. $\alpha + \psi = \varphi$. And so we have proven that $ Af([0,T], \R^n) \oplus^a C^1_{0,0}([0,T], \R^n) = C^1([0,T], \R^n)$. Since $dim Af([0,T], \R^n) < + \infty$ the subspace $Af([0,T], \R^n)$ is complete, since  $C^1_{0,0}([0,T], \R^n)$ is closed in $C^1([0,T], \R^n)$ which is complete, when obtain the announced conclusion (cf. \cite{La2}, Corollary 1.5, p.388)
\end{proof}
From this lemma, we can write
\begin{equation}\label{eq519}
 C^1([0,T],\R^n) =  Af([0,T], \R^n) \times C^1_{0,0}([0,T], \R^n).
\end{equation}
Using Lemma \ref{lem510} for $V = Af([0,T], \R^n)$ and $W =  C^1_{0,0}([0,T], \R^n)$, and denoting $\mathcal{S}_1 : (Af([0,T], \R^n))^* \times ( C^1_{0,0}([0,T], \R^n))^* \rightarrow (C^1([0,T], \R^n))^*$, defined by $\mathcal{S}_1 (\chi, \Lambda) := \chi \oplus \Lambda$, we obtain
\begin{equation}\label{eq520}
\mathcal{S}_1 \in Isom ((Af([0,T], \R^n))^* \times ( C^1_{0,0}([0,T], \R^n))^*, (C^1([0,T], \R^n))^*).
\end{equation}
Using Lemma \ref{lem510} for $V = \R^n$ and $W = C^0([0,T], \R^n)$, and denoting $\mathcal{S}_2 : (\R^n)^* \times (C^0([0,T], \R^n))^* \rightarrow (\R^n \times C^0([0,T], \R^n))^*$, defined by $\mathcal{S}_2(\beta, \Theta) := \beta \oplus \Theta$, we obtain
\begin{equation}\label{eq521}
\mathcal{S}_2 \in Isom((\R^n)^* \times (C^0([0,T], \R^n))^*, (\R^n \times C^0([0,T], \R^n))^*).
\end{equation}
The operator $\mathcal{T} : C^1([0,T], \R^n) \rightarrow \R^n \times C^0([0,T], \R^n)$, defined by $\mathcal{T}(x) := (x(0),x')$, is a topological isomorphism. Hence (cf. \cite{Fr}, Theorem 4.13.4, p.173), its adjoint satisfies the following property.
\begin{equation}\label{eq522}
\mathcal{T}^* \in Isom((\R^n \times C^0([0,T], \R^n)^*, (C^1([0,T], \R^n))^*).
\end{equation}
We denote by $\mathcal{R}_F : (\R^n)^* \rightarrow \R^n$ the isomorphism of F. Riesz and Fr\'echet (\cite{Ru}, p. 81), and by $\mathcal{R}_M : (C^0([0,T], \R^n))^* \rightarrow NBV([0,T], \R^n)$ the topological isomorphism of F. Riesz and Markov (\cite{KF}, p. 365). We define the operator $(\mathcal{R}_F, \mathcal{R}_M) : (\R^n)^* \times (C^0([0,T], \R^n))^* \rightarrow \R^n \times NBV([0,T], \R^n)$, defined by $(\mathcal{R}_F, \mathcal{R}_M)(\beta, \Theta) := (\mathcal{R}_F(\beta),  \mathcal{R}_M (\Theta))$. We easily verify the following property.
\begin{equation}\label{eq523}
(\mathcal{R}_F, \mathcal{R}_M) \in Isom((\R^n)^* \times (C^0([0,T], \R^n))^* , \R^n \times NBV([0,T], \R^n)).
\end{equation}
Now we can establish the following result.
\begin{lemma}\label{lem512} There exists an inner product on $(C^1_{0,0}([0,T], \R^n)^* \times (C^1_{0,0}([0,T], \R^n)^*$ which is continuous with respect to the usual topology of the product space.
\end{lemma}
\begin{proof} We consider the operator
\begin{equation}\label{eq524}
in : (C^1_{0,0}([0,T], \R^n)^* \rightarrow (Af([0,T], \R^n))^* \times (C^1_{0,0}([0,T], \R^n)^*, \; in(\Lambda) := (0, \Lambda).
\end{equation}
This operator is linear continuous and injective. We introduce the following operator: 
\begin{equation}\label{eq525}
\Gamma := (\mathcal{R}_F, \mathcal{R}_M) \circ \mathcal{S}_2^{-1} \circ (\mathcal{T}^*)^{-1} \circ \mathcal{S}_1 \circ in.
\end{equation}
From (\ref{eq520}), (\ref{eq521}), (\ref{eq522}), (\ref{eq523}) and (\ref{eq524}), we obtain that $\Gamma$ is linear continuous and injective. Now we build the operator $\Delta$  from $(C^1_{0,0}([0,T], \R^n)^* \times (C^1_{0,0}([0,T], \R^n)^*$ into $(\R^n \times NBV([0,T], \R^n) \times (\R^n \times NBV([0,T], \R^n))$ by setting 
\begin{equation}\label{eq526}
\Delta := (\Gamma \circ pr_1, \Gamma \circ pr_2)
\end{equation}
where $pr_1$ and $pr_2$ are the projections of $(C^1_{0,0}([0,T], \R^n)^* \times (C^1_{0,0}([0,T], \R^n)^*$. From (\ref{eq525}), we obtain that $\Delta$ is continuous on $(C^1_{0,0}([0,T], \R^n)^* \times (C^1_{0,0}([0,T], \R^n)^*$. \\
We consider the inner product $(\cdot \mid \cdot)_0$ on $\R^n \times NBV([0,T], \R^n)$ defined by 
\begin{equation}\label{eq527}
((\xi_1,g_1) \mid (\xi_2,g_2))_0 := (\xi_1 \mid \xi_2)_{\R^n} + \int_0^T(g_1(t) \mid g_2(t))_{\R^n} \; dt + (g_1(T) \mid g_2(T))_{\R^n}.
\end{equation}
This inner product is continuous withe respect to the usual norm of  \\
$\R^n \times NBV([0,T], \R^n)$. The functional $( \cdot \mid \cdot) := ( \cdot \mid \cdot)_0 \circ \Delta$ is an inner product on $(C^1_{0,0}([0,T], \R^n)^* \times (C^1_{0,0}([0,T], \R^n)^*$ which is continuous as a composition of continuous mappings.
\end{proof}
\subsection{Proof of Theorem \ref{th51}} {\bf Conclusion (I)}. 
Our strategy is to use Corollary 3.2. We start by doing the dictionary between the notation of Corollary \ref{cor32} and the notation of Theorem \ref{th51}. Let $\X := C^1([0,T], \R^n)$, $\A := C^1_{a_0,a_T}([0,T], \R^n)$, $\S := C^1_{0,0} ([0,T], \R^n)$, $\G := C^1([0,T], M)$, $f^0(x,\pi) = J(x,\pi) = \int_0^T L(t, x(t), x'(t), \pi) \, dt$, $g_i^0(x,\pi) = G_i(x, \pi) = \int_0^T \mathfrak{g}_i(t, x(t), x'(t), \pi)  dt$ when $1 \leq i \leq k$, $h^0_j(x, \pi) = H_j(x, \pi) =$\\
$ \int_0^T \mathfrak{h}_j(t, x(t), x'(t), \pi) \,  dt$ when $1 \leq j \leq {\ell}$, and $\underline{z}(\pi) = x(\pi)$. \\
Now we consider the assumptions of Corollary \ref{cor32}. ($A^0${\sc dua}) is fulfilled by using Lemma \ref{lem512}. ($A^0${\sc sol}1) is a consequence of (B{\sc sol}1), and ($A^0${\sc sol}2) is a consequence of (B{\sc sol}2).\\
To show that ($A^0${\sc fon}1) is fulfilled, we want to use Lemma \ref{lem55}, hence we ought to prove that the assumptions of Lemma \ref{lem55} are fulfilled. In the proof of Lemma \ref{lem59}, we have noted that (B{\sc sol}1), (B{\sc int}1), (B{\sc int}2) and (B{\sc int}3) ensure that the conditions (i-v) of Lemma \ref{lem55} are fulfilled and consequently the conclusions (a), (b) and (c) of Lemma \ref{lem55} hold. Hence ($A^0${\sc fon}1) results from the conclusions (b) and (c) of Lemma \ref{lem55}. ($A^0${\sc con}1) results from the conclusion (a) of lemma \ref{lem55}. ($A^0${\sc con}2) results from (B{\sc con}) and of Lemma \ref{lem58}. Then we can use the conclusion ($\alpha$) of Corollary \ref{cor32} which permits to ensure that the conclusion (I) of Theorem \ref{th51} is proven.\\
{\bf Conclusion (II)}. In the proof of the conclusion (I), we have yet proven that ($A^0${\sc sol}1), ($A^0${\sc sol}2), ($A^0${\sc fon}1), ($A^0${\sc con}1) and ($A^0${\sc con}2) are fulfilled. Replacing (B{\sc sol}2) by (B{\sc sol}2-bis), ($A^0${\sc sol}2-bis) is fulfilled, and we can use the conclusion ($\beta$) of Corollary \ref{cor32} to obtain the conclusion (II) of Theorem \ref{th51}.\\
{\bf Conclusion (III)}. After the proofs of the previous conclusions, we know that the assumptions (i), (ii), (iii), (iv) and (v) of Lemma \ref{lem55} are fulfilled. Note that (B{\sc int}3-bis) implies that assumption (vi) of Lemma \ref{lem55} holds. Also note that (B{\sc int}4) implies that the assumption (vii) of Lemma \ref{lem55} holds. Hence we can use the conclusion (d) of Lemma \ref{lem55}. From this conclusion (d) we deduce that ($A^0${\sc fon}2) holds. Note that (B{\sc sol}2-ter) implies ($A^0${\sc sol}2-ter). Hence we can use the conclusion ($\gamma$) of Corollary \ref{cor32} to obtain the conclusion (III) of Theorem \ref{th51}.
                                                                                                                                                                          
%

\end{document}